\documentclass[english]{smfart}
\usepackage[final]{microtype}
\usepackage[paperwidth=182mm,paperheight=257mm,text={144mm,210mm},centering]{geometry}
\usepackage{enumitem}
\setlist[enumerate]{left=\parindent,label=\textup{(\arabic*)}, ref=\textup{\arabic*}}
\setlist[itemize]{left=\parindent}

\usepackage{hyperref,amsthm,tikz-cd,xcolor}

\usepackage[theoremfont]{stickstootext}
\usepackage[frenchmath,smallerops,stix2,upint]{newtxmath}
\tikzcdset{arrow style=tikz, diagrams={>={Straight Barb[scale=0.7]}}} 
\DeclareFontEncoding{FMX}{}{}
\DeclareFontSubstitution{FMX}{futm}{m}{n}
\DeclareSymbolFont{fourier}{FMX}{futm}{m}{n}
\let\sum\relax\DeclareMathSymbol{\sum}{\mathop}{fourier}{80}
\let\prod\relax\DeclareMathSymbol{\prod}{\mathop}{fourier}{81}

\usepackage[cal=stixtwoplain,scr=rsfso]{mathalpha}

\renewcommand{\setminus}{\mathbin{\rule[0.2em]{0.67em}{0.12em}}}%
\renewcommand{\mathbb}{\mathbf}
\renewcommand{\emptyset}{\varnothing}
\renewcommand{\leq}{\leqslant}
\renewcommand{\geq}{\geqslant}
\renewcommand{\subset}{\subseteq}
\renewcommand{\supset}{\supseteq}
\newcommand{\Hom}{\mathcal{H}\!\textit{om}}

\newcommand{\capprod}{\mathbin{
  \tikz[baseline=\the\dimexpr\fontdimen22\textfont2\relax, line width=0.12ex] {
    \draw (0.12,0.12) arc[start angle=180, end angle=0, x radius=0.28em, y radius=0.4em];
  }
}}

\theoremstyle{remark}
\newtheorem{remark}{Remark}[subsection]
\newtheorem*{remark*}{Remark}

\theoremstyle{plain}
\newtheorem{theorem1}{Theorem}
\newtheorem{theorem}[remark]{Theorem}

\newtheorem{lemma}[remark]{Lemma}

\newtheorem*{proposition*}{Proposition}
\newtheorem{corollary}[remark]{Corollary}

\theoremstyle{definition}
\newtheorem{definition}[remark]{Definition}

\numberwithin{equation}{remark}
\newtheoremstyle{void} {.5\baselineskip plus .2\baselineskip minus .2\baselineskip} {.5\baselineskip plus .2\baselineskip minus .2\baselineskip} {\normalfont} {} {\bfseries} {.}
{5pt plus 1pt minus 1pt} 
{\thmname{#1 }\thmnumber{#2}\thmnote{.~{#3}}}
\theoremstyle{void}
\newtheorem{situation}[remark]{}
\newtheorem{method}[remark]{Method}

\usepackage[doi=false,isbn=false,sorting=nyt,url=false,eprint=false,backref=true,
style=alphabetic
]{biblatex}


\title[Asymptotic Singular Betti Bounds]{Asymptotic Betti bounds for hypersurfaces in a singular variety}

\author{Xuanyu Pan}
\address{Beijing Linear Stack Technologies Co., Ltd.}
\email{panxuanyu1985@163.com}
\author{Dingxin Zhang}
\address{Center for Mathematics and Interdisciplinary Sciences, Fudan
  University; and  Shanghai Institute for Mathematics and
  Interdisciplinary Sciences (SIMIS), Shanghai 200433, China}
\email{dingxinzhang@fudan.edu.cn}

\author{Xiping Zhang} \address{School of Mathematical Sciences and Key
  Laboratory of Intelligent Computing and Applications (Ministry of
  Education), Tongji University, Shanghai 200092, China}
\email{xzhmath@gmail.com}

\begin{document}
\begin{abstract}
We show that for any degree \(d\) hypersurface \(Y \subset X\) in a
possibly singular projective variety \(X\), the total Betti number of
\(Y\) is bounded by \(3\deg(X)\cdot d^n + C\cdot d^{n-1}\) for some
explicit constant \(C > 0\) independent of \(d\) and \(Y\).  When
\(X\) is a local complete intersection, the bound improves to
\(\deg(X)\cdot d^n + C\cdot d^{n-1}\).  In this case, the bound is
asymptotically sharp.  Similar bounds are also established for
general constructible sheaves.
\end{abstract}

\maketitle
\tableofcontents
\section{Introduction}
\renewcommand{\theremark}{\thesection.\arabic{remark}}

\begin{situation}
Let \(k\) be an algebraically closed field of arbitrary
characteristic.  (Topologists may take
\(k = \mathbb{C}\).)  Let \(X \subset \mathbb{P}^N\) be a possibly
singular projective variety over \(k\).
The main aim of this paper is to establish a uniform upper bound for
the total Betti number
\[
B(Y,\mathbb{F}_{\ell}) = \sum_{i} \dim \mathrm{H}^i(Y;\mathbb{F}_{\ell}),
\]
where
\begin{itemize}
\item \(\ell\) is a prime number invertible in \(k\),
\item \(\mathbb{F}_{\ell}\) is the finite field with \(\ell\) elements,
\(\mathrm{H}^{\ast}(-, \mathbb{F}_{\ell})\) denotes étale
cohomology, and
\item \(Y \subset X\) is the intersection of \(X\) with a
degree \(d\) hypersurface in \(\mathbb{P}^N\).
\end{itemize}

We are particularly interested in the asymptotic behavior of
\[
B_X(d) \coloneq \sup \left\{ B(Y,\mathbb{F}_{\ell}) : Y \text{ is a degree } d \text{ hypersurface in } X \right\}
\]
as \(d \to \infty\).

When \(k = \mathbb{C}\), the étale cohomology groups above are
isomorphic to the singular cohomology of the corresponding complex
analytic space with coefficients in \(\mathbb{F}_{\ell}\).  In
particular,
\[
B(Y,\mathbb{F}_{\ell}) = \sum_{i} \dim_{\mathbb{F}_{\ell}} \mathrm{H}^i(Y^{\mathrm{an}}; \mathbb{F}_{\ell})
\geq \sum_{i} \dim_{\mathbb{Q}} \mathrm{H}^i(Y^{\mathrm{an}}; \mathbb{Q}).
\]
Thus, bounds for \(B(Y, \mathbb{F}_{\ell})\) truly
concern the topology of algebraic varieties.

Beyond its intrinsic topological interest, a high-quality uniform
upper bound has applications in the arithmetic of algebraic varieties
and in analytic number theory.
\end{situation}

\begin{situation}\label{situation:pn-case}
When \(X = \mathbb{P}^N\), this problem was studied by Milnor~\cite{milnor_betti} and
Thom~\cite{thom_homology-of-real-varieties} for complex varieties (for singular cohomology
with coefficients in \(\mathbb{Q}\) or \(\mathbb{F}_2\)), and by Katz~\cite{katz_betti} in general.  Many others have contributed as well.  In the context of random real
complete intersections, a related question was considered by Ancona~\cite{ancona_topology-random-real-ci}.
More recently, Hu and Teyssier~\cite{hu-teyssier_betti} studied the asymptotics of Betti numbers
from a different perspective, using ramification theory.

In~\cite{maxim-paunescu-tibar_vanishing-cohomology-hypersurfaces} and
\cite{wan-zhang_betti-number-bounds-for-varieties-and-exponential-sums0},
asymptotically optimal bounds were established.  There exists a
constant \(C\) such that for any degree \(d\) hypersurface
\(Y \subset \mathbb{P}^N\), we have
\[
B(Y,\mathbb{F}_{\ell}) \leq d^N + C d^{N-1}.
\]
Moreover, the constant \(C\) can be made quite explicit by the methods
of loc.\ cit. In addition, in
\cite{wan-zhang_betti-number-bounds-for-varieties-and-exponential-sums0},
we also obtain asymptotically optimal bounds for the Betti numbers of
complete intersections in affine or projective space.
\end{situation}

\begin{situation}[Betti number bounds]
In this paper, we establish analogous bounds for hypersurfaces in an
\emph{arbitrary} projective variety \(X\).  When \(X\) is a local
complete intersection, the quality of our upper bound matches that in
the projective space case.  For varieties with arbitrary
singularities, the bound is slightly weaker, but nevertheless remains
remarkably sharp.
\end{situation}

\begin{theorem1}[\(\Leftarrow\) Theorem~\ref{theorem:arbitrary}]
\label{theorem1}
Suppose \(X \subset \mathbb{P}^{N}\) is an \(n\)-dimensional
projective variety.  Then there exists a constant \(C > 0\) such that,
for any degree \(d\) hypersurface \(Y \subset X\),
\[
B(Y, \mathbb{F}_{\ell}) \leq 3 \deg(X) \cdot d^n + C\cdot d^{n-1}.
\]
Here, \(\deg(X)\) denotes the degree of \(X \subset \mathbb{P}^N\).
\end{theorem1}

\begin{theorem1}[\(\Leftarrow\) Theorem~\ref{theorem:lci}]
\label{theorem2}
Suppose \(X \subset \mathbb{P}^{N}\) is an irreducible projective
variety of dimension \(n\) with at worst \emph{local complete intersection}
singularities.  Then there exists a constant \(C > 0\) such that, for
any degree \(d\) hypersurface \(Y \subset X\),
\[
B(Y, \mathbb{F}_{\ell}) \leq \deg(X) \cdot d^n + C\cdot d^{n-1}.
\]
\end{theorem1}

\begin{remark}
\begin{enumerate}[wide]
\item In the theorems above, the coefficient \(\mathbb{F}_{\ell}\) may
be replaced by any \(\ell\)-adic coefficient field \(\Lambda\), such
as any finite field of characteristic \(\ell\), or finite extensions
of \(\mathbb{Q}_{\ell}\), or \(\overline{\mathbb{Q}}_{\ell}\).
\item In the main text, we shall prove a generalization of
Theorem~\ref{theorem1} for any plain constructible sheaves on a
variety, and a generalization of Theorem~\ref{theorem2} for lisse
sheaves on a local complete intersection.
\end{enumerate}
\end{remark}

Theorems~\ref{theorem1} and~\ref{theorem2} provide precise asymptotic
estimates for the Betti numbers of degree \(d\) hypersurfaces in an
\(n\)-dimensional projective variety \(X\) as \(d \to \infty\).  When
\(X\) is an arbitrary projective variety, Theorem~\ref{theorem1}
shows that
\begin{equation}
\label{eq:asymp-general}
B_X(d) \coloneq \sup \left\{ B(Y, \mathbb{F}_{\ell}) : Y \text{ is a degree } d
\text{ hypersurface in } X \right\}
\asymp d^n \quad \text{as } d \to \infty.
\end{equation}
If \(X\) is a purely \(n\)-dimensional irreducible local complete
intersection, Theorem~\ref{theorem2} refines this to the sharper
asymptotic
\begin{equation}
\label{eq:asymp-lci}
B_X(d) \sim \deg(X) \cdot d^n \quad \text{as } d \to \infty.
\end{equation}
In fact, for a \emph{general} degree \(d\) hypersurface
\(Y \subset X\) where \(X\) is a local complete intersection, one
has an explicit estimate
\[
|B(Y, \mathbb{F}_{\ell}) - \deg(X) \cdot d^n | \leq C \cdot d^{n-1}.
\]
See Lemma~\ref{lemma:middle-general-estimate}.

The asymptotic relation \eqref{eq:asymp-lci} generalizes the
classical bounds for the Betti numbers of degree \(d\) hypersurfaces
in \(\mathbb{P}^N\) (i.e., \(\deg(X) = 1\)), as discussed in
\ref{situation:pn-case}.

Theorems~\ref{theorem1} and~\ref{theorem2} are deduced from a more
general result regarding the Betti numbers of \emph{perverse} sheaves
restricting to a hypersurface, although, in the context of
Theorem~\ref{theorem1}, the shifted constant sheaf
\(\mathbb{F}_{\ell,X}[n]\) on an \(n\)-dimensional singular variety is
\emph{not} perverse in general!

\begin{theorem1}[\(\Leftarrow\) Theorem~\ref{theorem:main}]
\label{theorem3}
Let \(\mathbb{P}\) be a smooth projective variety.  Let
\(\mathcal{P}\) be a perverse \(\mathbb{F}_{\ell}\)-sheaf or a
perverse \(\overline{\mathbb{Q}}_{\ell}\)-sheaf on \(\mathbb{P}\),
whose support has dimension \(n\).  Then there exist explicit
constants \(A\) and \(B\) such that for any degree \(d\) hypersurface
\(Y \subset \mathbb{P}\), we have
\[
\sum_i \dim \mathrm{H}^i(Y; \mathcal{P}|_Y) \leq A \cdot d^n + B \cdot d^{n-1}.
\]
\end{theorem1}

Therefore, we have the following asymptotic bound:
\[
B_{\mathbb{P}}(\mathcal{P}, d)
\coloneq \sup\left\{
\sum_i \dim \mathrm{H}^i(Y,\mathcal{P}|_Y) : Y \text{ is a degree } d \text{ hypersurface in } \mathbb{P}
\right\}
\asymp d^n \quad \text{as } d \to \infty,
\]
where \(n = \dim \operatorname{Supp}\mathcal{P}\).

In Theorem~\ref{theorem:main}, the constant \(A\) will be described
explicitly in terms of certain refined invariants of \(\mathcal{P}\).
The constant \(B\) is slightly more complicated than \(A\)
but is ultimately determined by
some characteristic numbers associated to \(\mathcal{P}\).

\begin{situation}[Uniformity in \(\ell\)]
For complex algebraic geometers and topologists, the
above summarizes the contents of the paper that might
interest them.  Analytic number theorists will probably also
find Theorems~\ref{theorem1}, \ref{theorem2}, and \ref{theorem3}
useful for certain applications, such as in equidistribution theorems.

However, for arithmetic algebraic geometers, there is a more delicate
issue: should the constants appearing in these bounds depend on the
auxiliary prime~\(\ell\)? Motivated by the long-standing conjectures
on the independence of~\(\ell\), it is natural to ask whether our
bounds are uniform in~\(\ell\).  We answer this question
affirmatively: all constants in the theorems can be taken to be
independent of~\(\ell\).
\end{situation}

\begin{theorem1}[\(\Leftarrow\) Theorem~\ref{theorem:independent-of-ell-for-main}, Corollary~\ref{corollary:independent-of-ell-for-arbitrary}]
The constants in Theorems~\ref{theorem1} and~\ref{theorem2} can be
chosen to be independent of~\(\ell\). If \(k\) is an algebraic
closure of a finite field, and \(\mathcal{P}\) is a
\(\overline{\mathbb{Q}}_{\ell}\)-perverse sheaf that is part of a
compatible system, then the constants \(A\) and \(B\) in
Theorem~\ref{theorem3} can also be chosen to be independent of~\(\ell\).
\end{theorem1}

\begin{situation}[Organization]
To keep this introduction short, we defer the technical discussion on
existing methods for bounding Betti numbers and the novel aspects of
this work to \S\ref{sec:howto}.  Section~\ref{sec:preliminaries}
provides the necessary background on perverse sheaves and
characteristic classes.  In Section~\ref{sec:proper-degeneration}, we
prove the key technical result, the Proper Degeneration
Lemma~\ref{lemma:proper-degeneration}.  Section~\ref{sec:main-theorem}
establishes the main theorem, giving bounds for the Betti numbers of
perverse sheaves restricted to hypersurfaces.
Section~\ref{sec:applications} applies these bounds to lisse sheaves.
Finally, Section~\ref{sec:l-independence} examines the dependence of
our estimates on the coefficient field and proves uniformity in
\(\ell\).
\end{situation}

\begin{situation}[Acknowledgment]
We are grateful to Michele Ancona and Junyan Cao for raising the
question of asymptotic upper bounds for hypersurfaces in arbitrarily
singular varieties.  Their question motivated this work.  We thank
Haoyu Hu, Daqing Wan, and Yigeng Zhao for valuable conversations and
helpful suggestions.  D.~Zhang is partially supported by the National
Key R\&D Program of China (No.~2022YFA1007100).  X.~Zhang is supported
by National Natural Science Foundation of China (No.~12201463)

\end{situation}

\section{How to obtain Betti bounds?}
\label{sec:howto}

We briefly survey existing methods for bounding Betti numbers,
analyzing their strengths and limitations.  We then explain what's new
in our approach.  Readers not interested in these technical
discussions may wish to proceed directly to
\ref{situation:conventions}

\begin{method}[Chain-level counting \cite{bombieri_exponential-sums-1,bombieri_exponential-sums-2,adolphson-sperber_degree-l-function}]
\label{method:chain-level}
In the Morse theory approach used by Milnor and Thom, the strategy is
to construct a Morse function on \(X\) and estimate Betti numbers by
counting the critical points, which can be bounded via Bézout's
theorem.  At first glance, this method appears too topological to
extend to varieties over fields of positive characteristic, and, as we
now know, the bounds given by Milnor and Thom are non-optimal even for
complex varieties.  Nevertheless, the underlying ``chain-level
counting'' philosophy can be adapted to the framework of Dwork theory
and rigid cohomology.

By a standard reduction (``Cayley's trick''), the problem of
estimating the Betti numbers of a variety defined by a regular
function can be reduced to estimating the Betti numbers of certain isocrystals associated to
certain exponential sums.

For exponential sums over an algebraic torus,
Bombieri~\cite{bombieri_exponential-sums-2} and
Adolphson--Sperber~\cite{adolphson-sperber_degree-l-function}
construct explicit chain complexes that compute the associated
cohomology, along with a detailed description of the Frobenius
action at the chain level.  This explicitness allows them to estimate
the number of Frobenius eigenvalues with small \(p\)-adic slopes
directly from the chain complex.  Because there are a priori upper bounds
on the slopes of Frobenius acting on cohomology, and since the
chain-level Frobenius is so well understood in this context, one obtains
sharp and effective upper bounds for the Betti numbers.

The primary limitation of this method is its dependence on detailed knowledge
of the infinite matrix representing the chain-level Frobenius, which is so
far achievable only for certain special F-isocrystals over tori or Shimura
curves of PEL type, where excellent Frobenius liftings are available.
\end{method}

\begin{method}[Euler characteristic trick \cite{katz_betti,zhang_betti}]
\label{method:euler-characteristic}
In \cite{katz_betti}, Katz used the weak Lefschetz theorem to develop
a method that reduces estimating Betti numbers to the problem of
bounding Euler characteristics.  The essential idea is that if one has
a uniform upper bound for the Euler characteristic for \emph{all}
hypersurfaces, and the \(n\)-dimensional affine variety \(X\) is
sufficiently mild in its singularities so that the weak Lefschetz
theorem applies, then the difference between the Euler characteristics
of \(X\) and a hyperplane section \(X \cap L\) gives an upper bound
for the middle Betti number of \(X\).  The remaining Betti numbers can
then be controlled by induction using the weak Lefschetz theorem. (We
have re-implemented this strategy in Lemma~\ref{lemma:katz}.)

A prerequisite for this method is the existence of a uniform,
effective upper bound for the Euler characteristic.  In loc.~cit.,
Katz obtained such bounds by using results of Bombieri and
Adolphson--Sperber, who employed Dwork theory (\(p\)-adic methods) to
estimate the Euler characteristic.  Additionally, a certain perversity
condition is required to ensure that the weak Lefschetz theorem
applies; see Lemma~\ref{lemma:weak-lefschetz} and
Lemma~\ref{lemma:gysin}.

Beside the applicability of weak Lefschetz, the above prerequisite on
Euler characteristic is also the main limitation of this approach.  In
the cases Katz treated, such bounds were established via Dwork theory
for tori.  In turn, this boils down to establishing a sharp lower
bound for the Newton polygon associated to the Fredholm determinant of
the chain-level Frobenius operator discussed in
Method~\ref{method:chain-level}.
\end{method}

\begin{method}[Perverse degeneration \cite{wan-zhang_betti-number-bounds-for-varieties-and-exponential-sums0}]
\label{method:degeneration}
When uniform upper bounds for the Euler characteristic are
unavailable, the above method cannot be applied.  However, if the weak
Lefschetz theorem remains valid, one can instead focus on estimating
the dimension of the cohomology that does not arise from hyperplane
sections.  For example, if the variety has local complete intersection
singularities, the shifted constant sheaf is perverse, so the problem
reduces to bounding the middle-dimensional cohomology.

In this context, the so-called perverse degeneration lemma
\cite{katz_perverse-origin,wan-zhang_betti-number-bounds-for-varieties-and-exponential-sums0}
is very useful.  This lemma allows us to control the middle-dimensional
cohomology of an affine hypersurface by relating it to that of a nearby
hypersurface.  Once the middle-dimensional cohomology is bounded,
the weak Lefschetz theorem, together with perversity, allows us to
handle the cohomology in the remaining degrees by induction.

By combining the weak Lefschetz theorem with the perverse degeneration
lemma, one arrives at the following strategy for obtaining Betti number
bounds for hypersurfaces:
\begin{itemize}
\item For a \emph{general} hypersurface \(Y\), the Euler characteristic
usually can be computed explicitly (see Lemma~\ref{lemma:generic-euler-characteristic-estimate}).
The weak Lefschetz theorem then allows
one to determine the middle cohomology group precisely for such \(Y\).
\item The perverse degeneration lemma provides uniform bounds for the
middle-dimensional cohomology of \emph{arbitrary} hypersurfaces, not only
the general ones.
\item Finally, applying the weak Lefschetz theorem inductively (when possible)
enables control of the remaining (non-middle) cohomology groups, provided
that appropriate perversity conditions are satisfied.
\end{itemize}

Method~\ref{method:degeneration} works well if the constant sheaf is a
perverse sheaf, but is insufficient for Theorem~\ref{theorem1}.  In
this case, the shifted constant sheaf is generally not perverse, nor
is its restriction to an arbitrary hypersurface.  Consequently, there
may remain cohomological degrees that are not accessible
simultaneously via the weak Lefschetz theorem and the perverse
degeneration lemma.
\end{method}

\begin{situation}[New methods in this paper]
The approach in this paper is similar to
Method~\ref{method:degeneration}, but introduces two new ingredients.

The first is more conceptual: rather than proving bounds for constant
sheaves, we work with general \emph{perverse} sheaves, see
Theorem~\ref{theorem:main}.  This shift in perspective allows us to
apply the perverse weak Lefschetz theorem to control most cohomology
groups.  Bounds for constant sheaves are then deduced from the
perverse case via dévissage.  See Theorem~\ref{theorem:arbitrary}.

The second new ingredient is more technical: the Proper Degeneration
Lemma~\ref{lemma:proper-degeneration}.  When perversity conditions
hold, this lemma is as effective as the perverse degeneration lemma,
but it has the advantage of applying to cohomology in \emph{all}
degrees, not just the middle one.  However, this greater generality
leads to less optimal bounds.
\end{situation}

\section{Preliminaries}
\label{sec:preliminaries}

\renewcommand{\theremark}{\thesubsection.\arabic{remark}}

\subsection{Notation and conventions}

We begin by introducing the notation and conventions used
throughout this paper.

\begin{situation}
\label{situation:conventions}
Throughout this paper, except in \S\ref{sec:l-independence}, we work
over an algebraically closed field \(k\).  A \emph{variety} over \(k\)
is defined to be a separated scheme of finite type over \(k\).  All
morphisms between varieties are \(k\)-morphisms.

Varieties here are not necessarily assumed to be irreducible or
reduced.  Since the étale topology of a variety \(V\) is the same as
that of its reduction \(V_{\mathrm{red}}\), the property of being
reduced is mostly irrelevant.
\end{situation}

\begin{situation}
Except in \S\ref{sec:l-independence}, we fix a prime number \(\ell\)
invertible in \(k\), and let \(\Lambda\) be an \(\ell\)-adic
coefficient field.  For example, \(\Lambda\) may be a finite extension
of \(\mathbb{F}_{\ell}\), a finite extension of \(\mathbb{Q}_{\ell}\),
or \(\overline{\mathbb{Q}}_{\ell}\).  For a variety \(V\) over \(k\),
let \(D^b_c(V,\Lambda)\) denote the bounded derived category of
constructible complexes of \(\Lambda\)-sheaves on \(V\) (see
\cite[\S1.1]{deligne_weil2}).

A constructible complex \(\mathcal{F} \in D^b_c(V,\Lambda)\) is said
to be \emph{lisse} if, for any \(i \in \mathbb{Z}\),
\(\mathcal{H}^i(\mathcal{F})\) is a lisse sheaf (or in a topologist's
language, a ``local system'') of \(\Lambda\)-modules.  If \(\Lambda\)
is finite, then a lisse sheaf of \(\Lambda\)-modules is just a locally
constant sheaf of \(\Lambda\)-modules with respect to the étale
topology.  If \(\Lambda\) is an \(\ell\)-adic field, we refer the
reader to \cite[(1.1.1b)]{deligne_weil2} for the definition of a lisse
sheaf of \(\Lambda\)-modules.

Let \(f\colon X \to Y\) be a morphism of varieties.  The following
standard six operations are available on \(D^b_c(-, \Lambda)\):
\begin{align*}
  Rf_{!}, Rf_{\ast} &\colon D^b_c(X, \Lambda) \to D^b_c(Y, \Lambda), \\
  f^{\ast}, Rf^{!} &\colon D^b_c(Y, \Lambda) \to D^b_c(X, \Lambda),
\end{align*}
as well as the internal operations \(\otimes\) and \(R\Hom\) on
\(D^b_c(-, \Lambda)\).

We also have the \emph{biduality functor}
\(\mathbb{D}_{X} \colon D^b_c(X, \Lambda)^{\mathrm{opp}} \to D^b_c(X, \Lambda)\),
defined by
\[
\mathbb{D}_{X}(-) = R\Hom(-, Ra^! \Lambda),
\]
where \(a\colon X \to \operatorname{Spec}k\) is the structural
morphism.

\end{situation}

\begin{situation}[Miscellaneous notation]
Let \(\mathcal{F} \in D^{b}_{c}(V,\Lambda)\) be a constructible
complex of \(\Lambda\)-sheaves on \(V\).  We define
\[
B(V,\mathcal{F}) = \sum_{m \in \mathbb{Z}} \dim_{\Lambda}
\mathrm{H}^{m}(V;\mathcal{F})
\]
to be the sum of the Betti numbers of \(V\) with coefficients
in \(\mathcal{F}\).

Let \(Z\) be a closed subvariety of \(V\).  We denote the cohomology
of \(\mathcal{F}\) relative to \(Z\) by
\[
\mathrm{H}^{\ast}(V, Z; \mathcal{F}) = \mathrm{H}^{\ast}(V, j_{!}j^{\ast}\mathcal{F}),
\]
where \(j\colon V \setminus Z \hookrightarrow V\) denotes the
inclusion of the open complement of \(Z\) in \(V\).

We write \(\mathcal{F}(m)\) for the \(m\)th Tate twist of
\(\mathcal{F}\), that is,
\(\mathcal{F}(m) = \mathcal{F} \otimes_{\Lambda} \Lambda(m)\).  Here,
\(\Lambda(-1)\) denotes
\(\mathrm{H}^{1}(\mathbb{G}_{\mathrm{m}};\Lambda)\), while
\(\Lambda(1) = \mathrm{Hom}_{\Lambda}(\Lambda(-1),\Lambda)\).  More
generally, for \(m \geq 1\), we define
\(\Lambda(m) = \Lambda(1)^{\otimes m}\), and for \(m \leq 1\),
\(\Lambda(m) = \Lambda(-1)^{\otimes (-m)}\).  In our geometric
setting, \(\Lambda(1)\) is simply a 1-dimensional \(\Lambda\)-vector
space, but when working in an arithmetic context it is necessary to
keep careful track of the Galois action given by the twist.
\end{situation}

\subsection{Perverse sheaves}
In this subsection, we present some properties of perverse sheaves
that will be used throughout this paper.

\begin{situation}
For basic facts about perverse sheaves, we refer the reader to
\cite{bbd} and
\cite{kiehl-weissauer_weil-conjectures-perverse-sheaves}.  The
\(t\)-structure on \(D^b_c(X, \Lambda)\) associated with the middle
perversity is denoted by
\(({}^{\mathrm{p}}D^{\leq0}_c(X,\Lambda),
{}^{\mathrm{p}}D^{\geq0}_c(X,\Lambda))\).  We use the notations
\({}^{\mathrm{p}}\mathcal{H}^n\), \({}^{\mathrm{p}}\tau_{\leq}\), and
similar conventions as in loc.~cit.  Constructible complexes in
\({}^{\mathrm{p}}D^{\leq0}_c(X, \Lambda)\) are called \emph{perverse
  connective}, and those in \({}^{\mathrm{p}}D^{\geq0}_c(X, \Lambda)\)
are called \emph{perverse coconnective}.  The duality functor
\(\mathbb{D}_{X}\) interchanges the connective and coconnective parts
(see \cite[2.2 and 4.3.1]{bbd}).
\end{situation}

To bound the total Betti numbers of a constructible complex,
we first make a straightforward reduction to the case of perverse
sheaves:

\begin{lemma}\label{lemma:trivial}
Let \(V\) be a variety and \(\iota \colon Z \to V\) a closed
subvariety.  Then for any \(\mathcal{F} \in D^b_c(V,\Lambda)\),
we have
\[
B(Z, \iota^{\ast}\mathcal{F}) \leq
\sum_{i\in \mathbb{Z}}
B\bigl(Z, \iota^{\ast}({}^{\mathrm{p}}\mathcal{H}^i(\mathcal{F}))\bigr).
\]
\end{lemma}

\begin{proof}
Given any constructible complex \(\mathcal{F}\), there exist integers
\(a \leq b\) such that
\({}^{\mathrm{p}}\mathcal{H}^i(\mathcal{F}) = 0\) unless \(a \leq i \leq b\).
We proceed by induction on the integer \(b-a\), which we refer to as the
length of \(\mathcal{F}\).

If \(b-a = 0\), then, up to shift, \(\mathcal{F}\) itself is a perverse
sheaf.  In this case, the result is immediate.

For the induction step, let \(b\) be the largest index such that
\({}^{\mathrm{p}}\mathcal{H}^b(\mathcal{F}) \neq 0\).
There is a distinguished triangle
\[
{}^{\mathrm{p}}\tau_{\leq b-1} \mathcal{F} \to \mathcal{F} \to {}^{\mathrm{p}}\mathcal{H}^b(\mathcal{F}),
\]
which, after applying \(\iota^{\ast}\), yields
\[
\iota^{\ast}({}^{\mathrm{p}}\tau_{\leq b-1} \mathcal{F}) \to \iota^{\ast} \mathcal{F}
\to \iota^{\ast}({}^{\mathrm{p}}\mathcal{H}^b(\mathcal{F})).
\]
Therefore,
\[
B(Z, \iota^{\ast} \mathcal{F}) \leq
B(Z, \iota^{\ast}({}^{\mathrm{p}}\tau_{\leq b-1} \mathcal{F})) +
B(Z, \iota^{\ast}({}^{\mathrm{p}}\mathcal{H}^b(\mathcal{F}))).
\]
The length of \({}^{\mathrm{p}}\tau_{\leq b-1} \mathcal{F}\) is
smaller than that of \(\mathcal{F}\), so we can apply the inductive
hypothesis to this term.  Thus, by induction, the lemma is proved.
\end{proof}

Next, recall the following fundamental fact:

\begin{theorem}[Artin's vanishing theorem]
\label{theorem:relative-affine-lefschetz}
Let \(\pi\colon X \to S\) be an affine morphism of varieties.
Let \(\mathcal{F} \in {}^{\mathrm{p}}D_c^{\geq 0}(X, \Lambda)\) be perverse
coconnective on \(X\).  Then \(R\pi_{!}\mathcal{F}\) is perverse
coconnective on \(S\).
\end{theorem}

\begin{proof}
See \cite[Exposé~XIV,~3.1]{sga4} or
\cite[Théorème~4.1.1]{bbd}.
\end{proof}

We then prove some elementary properties about how perversity changes under
some special operations.

\begin{lemma}\label{lemma:check-connectivity}
Let \(f\colon U \to V\) be an étale and surjective morphism.  Let
\(\mathcal{F} \in D^b_c(V, \Lambda)\).  Suppose
\(f^{\ast}\mathcal{F} \in {}^{\mathrm{p}}D^{\leq 0}_c(U, \Lambda)\)
(resp.\ in \({}^{\mathrm{p}}D^{\geq 0}_c(U, \Lambda)\), or is perverse).
Then \(\mathcal{F}\) lies in \({}^{\mathrm{p}}D^{\leq 0}_c(V, \Lambda)\)
(resp.\ in \({}^{\mathrm{p}}D^{\geq 0}_c(V, \Lambda)\), or is perverse).
\end{lemma}

\begin{proof}
We treat only the case of \({}^{\mathrm{p}}D_c^{\leq 0}\); the other
cases follow from Verdier duality and the fact that for an étale
morphism \(f\), we have \(f^! = f^*\).

Since \(f\) is surjective, for any generic point \(\eta\) of
\(\operatorname{Supp} \mathcal{H}^i(\mathcal{F})\), there exists a
point \(\eta_U\) lying above it.  The support
\(\operatorname{Supp} \mathcal{H}^i(f^* \mathcal{F})\) contains the
closure of \(\eta_U\), which has the same dimension as the closure of
\(\eta\), by the étale property.  Outside the closures of these
\(\eta_U\), the stalks of \(\mathcal{H}^i(f^*\mathcal{F})\) are zero.
Therefore,
\[
\dim \operatorname{Supp} \mathcal{H}^i(f^*\mathcal{F})
= \dim \operatorname{Supp} \mathcal{H}^i(\mathcal{F}) .
\]
Thus, if \(\mathcal{F}\) is connective in the perverse sense, so is
\(f^*\mathcal{F}\).
\end{proof}

\begin{lemma}\label{lemma:estimate-perv-dergee}
Let \(\mathcal{P}\) be a perverse sheaf on a variety \(V\).  Let
\(\iota\colon D \to V\) be the inclusion morphism of a Cartier
divisor.  Then
\({}^{\mathrm{p}}\mathcal{H}^{e}(\iota^{\ast}\mathcal{P}[-1]) = 0\)
for \(e \neq 0,1\), and
\({}^{\mathrm{p}}\mathcal{H}^{e}(i^{!}\mathcal{P}[1]) = 0\) for
\(e \neq 0,-1\).
\end{lemma}

\begin{proof}
The second statement follows from the first for
\(\mathbb{D}_{V}(\mathcal{P})\).  So it suffices to prove the first
statement.

Let \(j\colon U \to V\) be the inclusion morphism of the complement of
\(D\).  Then \(j\) is both affine and quasi-finite, hence \(j_{!}\) is
t-exact
(\cite[Corollary~III~6.2]{kiehl-weissauer_weil-conjectures-perverse-sheaves}).
Thus, \(j_{!}j^{\ast}\mathcal{P}\) is a perverse sheaf on \(V\).
Applying \({}^{\mathrm{p}}\mathcal{H}^{\ast}\) to the excision
triangle
\[
\mathcal{P}[-1] \to \iota_{\ast}\iota^{\ast}\mathcal{P}[-1] \to j_{!}j^{\ast}\mathcal{P},
\]
gives the following exact sequence
\[
{}^{\mathrm{p}}\mathcal{H}^{e-1}(\mathcal{P}) \to {}^{\mathrm{p}} \mathcal{H}^e(\iota^{\ast}\mathcal{P}[-1])
\to {}^{\mathrm{p}}\mathcal{H}^e(j_{!}j^{\ast}\mathcal{P}).
\]
Since \(\mathcal{P}\) and \(j_{!}j^{\ast}\mathcal{P}\) are perverse,
the result follows.
\end{proof}

\begin{lemma}\label{lemma:general-hyperplane-preserve-perv}
Let \(V \subset \mathbb{P}^N\) be a quasi-projective variety.
Let \(\mathcal{F}\) be a constructible complex on \(V\).
Then for a general hyperplane section \(\iota\colon L \hookrightarrow V\),
we have
\[
\iota^{\ast}({}^{\mathrm{p}}\mathcal{H}^{i}(\mathcal{F})[-1]) \simeq {}^{\mathrm{p}}\mathcal{H}^{i}(\iota^{\ast}\mathcal{F}[-1]).
\]
\end{lemma}

\begin{proof}
We proceed by induction on the length of \(\mathcal{F}\), as in the proof
of Lemma~\ref{lemma:trivial}.  If the
length is zero, then after a shift, we can assume \(\mathcal{F}\) is a
perverse sheaf.  By Lemma~\ref{lemma:estimate-perv-dergee},
\(\iota^{\ast}\mathcal{F}[-1]\) is perverse coconnective
and \(R\iota^{!}\mathcal{F}[1]\) is perverse connective.
Since \(L\) is general, the purity isomorphism holds; that is,
there is an isomorphism
\[
\iota^{\ast}\mathcal{F}[-1](-1) \xrightarrow{\sim} \iota^{!}\mathcal{F}[1].
\]
See \cite[Lemma~3.2]{wan-zhang_colevel}.  The claim then follows from
Lemma~\ref{lemma:estimate-perv-dergee}.

For general \(\mathcal{F}\), we first apply a shift and reduce to the
case \(\mathcal{F} \in {}^{\mathrm{p}}D^{\leq0}_{c}(V,\Lambda)\).
There is a distinguished triangle
\[
{}^{\mathrm{p}}\tau_{\leq-1}\mathcal{F} \to \mathcal{F}
\to {}^{\mathrm{p}}\mathcal{H}^{0}(\mathcal{F}).
\]
Applying \(\iota^{\ast}[-1]\) yields
\[
\iota^{\ast}({}^{\mathrm{p}}\tau_{\leq-1}\mathcal{F})[-1]
\to \iota^{\ast}\mathcal{F}[-1]
\to \iota^{\ast}({}^{\mathrm{p}}\mathcal{H}^{0}(\mathcal{F}))[-1].
\]
For general \(L\), the induction hypothesis shows that
\(\iota^{\ast}[-1]\) commutes with the perverse cohomology sheaves of
\({}^{\mathrm{p}}\tau_{\leq-1}(\mathcal{F})\), which are also the
perverse cohomology sheaves of \(\mathcal{F}\).  In particular,
\(\iota^{\ast}({}^{\mathrm{p}}\tau_{\leq-1}(\mathcal{F}))[-1]\) still
lies in \({}^{\mathrm{p}}D^{\leq-1}_{c}(L,\Lambda)\).  Therefore,
\[
\iota^{\ast}({}^{\mathrm{p}}\mathcal{H}^{0}(\mathcal{F}))[-1] =
{}^{\mathrm{p}}\mathcal{H}^{0}(\iota^{\ast}\mathcal{F}[-1]).
\]
This completes the proof.
\end{proof}



Finally, we recapitulate some weak Lefschetz theorems for
perverse sheaves.  The first one is just a consequence of Artin's vanishing theorem.

\begin{lemma}\label{lemma:weak-lefschetz}
Let \(X\) be a projective variety and suppose
\(\mathcal{F} \in {}^{\mathrm{p}}D_c^{\geq0}(X,\Lambda)\) is perverse
coconnective.  Then for any ample hypersurface \(Y\) in \(X\), the
restriction map
\[
\mathrm{H}^{m}(X;\mathcal{F}) \to \mathrm{H}^{m}(Y;\mathcal{F}|_{Y})
\]
is injective for \(m = -1\), and bijective for \(m \leq -2\).
\end{lemma}

\begin{proof}
Since \(X\setminus Y\) is affine, Artin's vanishing theorem
(Theorem~\ref{theorem:relative-affine-lefschetz}) implies
\(\mathrm{H}^{m}_{c}(X\setminus Y; \mathcal{F}) = 0\) for
\(m \leq -1\).  The lemma then follows from the standard long exact
sequence associated to the excision triangle.
\end{proof}

The second lemma is less straightforward.  Ultimately, it relies on
Deligne's generic base change theorem, which is used to apply
Artin's vanishing theorem.  As a result, it holds only for general
hyperplane sections.

\begin{lemma}[Perverse weak Lefschetz theorem {\cite[Lemma~3.8]{wan-zhang_colevel}}]
\label{lemma:gysin}
Let \(X\) be an algebraic variety over an algebraically closed field
\(k\), and let \(\iota \colon X \to \mathbf{P}^{N}\) be a quasi-finite
morphism.  Suppose
\(\mathcal{F} \in {}^{\mathrm{p}}D^{\leq0}_{c}(X,\Lambda)\) is
perverse connective.  Then, for a general hyperplane
\(B \subset \mathbf{P}^{N}\), the Gysin map
\[
\mathrm{H}^{i-2}_{\mathrm{c}}
\left(X \times_{\mathbf{P}^{N}} B; \mathcal{F}(-1)\right)
\longrightarrow
\mathrm{H}^{i}_{\mathrm{c}}(X; \mathcal{F})
\]
is surjective for \(i = 1\) and is an isomorphism for all \(i \geq 2\).
\end{lemma}

\begin{proof}[Sketch of proof]
Set \(\mathcal{G} = \mathbf{D}_{X}\mathcal{F}\).  Then
\(\mathcal{G} \in {}^{\mathrm{p}}D^{\geq0}_c(X,\Lambda)\).
By Deligne's perverse weak Lefschetz theorem
\cite[Corollary~A.5]{katz_affine-cohomological-transforms}, for a
general hyperplane \(B\), the restriction map
\begin{equation}
\label{eq:deligne-lef}
\mathrm{H}^i(X;\mathcal{G}) \to
\mathrm{H}^i(X\times_{\mathbf{P}^N}B;\mathcal{G}|_{B})
\end{equation}
is injective if \(i = -1\), and bijective if \(i \leq -2\).  Although
the statement in loc.~cit.\ is for perverse sheaves, the proof only
requires that the complex is coconnective, which is sufficient in
order to apply Artin's vanishing theorem.

Using Deligne's generic base change
(\cite[Corollarie~2.9]{deligne_finitude}), one can show that for
general \(B\), the morphism
\(\iota\colon X \times_{\mathbf{P}^N} B \to X\) also satisfies
\begin{equation}
\label{eq:gen-purity}
R\iota^{!}\mathcal{F} \simeq \mathcal{F}[-2](-1)
\end{equation}
(see, for example, \cite[Lemma~3.2]{wan-zhang_colevel} for details).

The lemma then follows by applying Verdier duality to
\eqref{eq:deligne-lef}, and then using \eqref{eq:gen-purity}.
\end{proof}

\subsection{Characteristic classes}

In addition to perverse sheaves, another technical tool required
for this paper is the theory of characteristic classes of constructible
complexes, developed by Takeshi Saito in the étale
setting~\cite{saito_takeshi_characteristic-cycle-and-singular-support}.
The purpose of this subsection is to recall some simple properties we will
need in what follows.

In \cite{saito_takeshi_characteristic-cycle-and-singular-support},
characteristic classes are defined for finite \(\Lambda\).  But
Umezaki, Yang and Zhao \cite[\S5.3]{umezaki-yang-zhao} observed that the
whole theory can be readily adapted to the case of
\(\overline{\mathbb{Q}}_{\ell}\)-coefficients.

\begin{situation}[Notation]
Let \(X\) be a variety over \(k\).  Denote by
\[
\mathrm{CH}_{\ast}(X) = \bigoplus_{i=0}^{\dim X} \mathrm{CH}_i(X)
\]
the Chow group of \(X\) \cite[\S1.3]{fulton_intersection-theory}.
In particular, there is a
\emph{fundamental cycle class} \([X] \in \mathrm{CH}_{\ast}(X)\)
\cite[\S1.5]{fulton_intersection-theory}.  For
\(\xi \in \mathrm{CH}_{\ast}(X)\), let \(\xi_j \in \mathrm{CH}_j(X)\)
denote the degree \(j\) component of \(\xi\).
We write \(\int_X \xi = \deg \xi_0\), where \(\deg\colon \mathrm{CH}_0(X) \to \mathbb{Z}\)
is the degree of a zero-cycle class \cite[Definition~1.4]{fulton_intersection-theory}.
If \(E\) is a vector bundle on \(X\), and \(P\) is a polynomial in the Chern
classes of \(E\), then for any \(\xi \in \mathrm{CH}_{\ast}(X)\), the
\emph{cap product} \(P \capprod \xi\) is defined as in
\cite[\S3.2]{fulton_intersection-theory}.
\end{situation}

\begin{situation}\label{situation:construction-cc}
Let \(X\) be a smooth variety over \(k\).  In the étale setting,
Takeshi
Saito~\cite[\S6.2]{saito_takeshi_characteristic-cycle-and-singular-support}
constructed a group homomorphism
\[
\mathrm{cc}_X \colon K_0(D^b_c(X,\Lambda)) \to \mathrm{CH}_{\ast}(X),
\]
which assigns to each constructible complex \(\mathcal{F}\) its
\emph{characteristic class} \(\mathrm{cc}_X(\mathcal{F})\) in the Chow
group.

We briefly recall the construction.  Assume \(X\) is a smooth variety
of pure dimension \(n\).  The cotangent bundle
\(T^{\ast}X = \operatorname{\mathit{Spec}}_X\mathrm{Sym}^{\bullet}(\Omega_{X/k}^{1})^{\vee}\)
admits a natural compactification by adding a ``hyperplane at
infinity'':
\[
T^{\ast}X \hookrightarrow \mathbb{P}(T^{\ast}X \times_{X} \mathbb{A}^{1}_{X}) = \operatorname{\mathit{Proj}}_{X} \mathrm{Sym}^{\bullet} (\Omega_{X}^{1} \oplus \mathscr{O}_{X})^{\vee}.
\]
Here,
\(
\mathbb{P}(T^{\ast}X) \simeq \mathbb{P}(T^{\ast}X \times_{X} \mathbb{A}^{1}_{X}) \setminus T^{\ast} X
\)
plays the role of this hyperplane at infinity.
Given a constructible complex \(\mathcal{F}\) on \(X\), denote its
\emph{characteristic cycle}
(which is an element in \(\mathrm{CH}_{n}(T^{\ast}X)\), see \cite[Definition~5.10]{saito_takeshi_characteristic-cycle-and-singular-support})
by \(\mathrm{CC}(\mathcal{F})\).  The projectivization
\[
\overline{\mathrm{CC}}(\mathcal{F}) = \mathbb{P}(\mathrm{CC}(\mathcal{F}) \times_{X} \mathbb{A}^{1}_{X}).
\]
is then an \(n\)-cycle
\(\overline{\mathrm{CC}}(\mathcal{F})\in \mathrm{CH}_{n}(\mathbb{P}(T^{\ast}X \times_{X} \mathbb{A}^{1}_{X}))\).
The \emph{characteristic class} \(\mathrm{cc}_X(\mathcal{F})\)
is then defined via the following formula
\[
\mathrm{cc}_X(\mathcal{F}) = c(\Omega^1_X) \capprod
\varpi_{\ast} \left( \sum_{i \geq 0} c_{1}(\mathscr{O}(1))^{i} \capprod \overline{\mathrm{CC}}(\mathcal{F}) \right) \in \mathrm{CH}_{\ast}(X),
\]
where \(c(\Omega^1_X)\) is the total Chern class of the cotangent
bundle,
\(\varpi\colon \mathbb{P}(T^{\ast}X \times_{X} \mathbb{A}^{1}_{X}) \to X\)
denotes the projection, and \(\mathscr{O}(1)\) is the dual of the
relative tautological bundle.

As a basic example, consider the constant sheaf \(\Lambda_X[n]\).  Its
characteristic cycle \(\mathrm{CC}(\Lambda_X[n])\) is the zero section
\(T^{\ast}_X X\) inside the cotangent bundle. In this case, we have
\[
\mathrm{cc}_X(\Lambda_X[n]) = c_{\ast}(\Omega^1_{X/k}) \capprod [X].
\]
\end{situation}

These characteristic classes allow the computation of Euler
characteristics of constructible complexes via the \emph{index formula}.
This result is due to
Brylinski--Dubson--Kashiwara~\cite{brylinski-dubson-kashiwara} in
characteristic zero (see also
\cite[(9.1.14)]{kashiwara-schapira_sheaves-on-manifolds}) and to
Takeshi Saito~\cite[Lemma~6.9(1),
Theorem~7.13]{saito_takeshi_characteristic-cycle-and-singular-support}
in the étale setting.

\begin{theorem}[Index Formula]\label{theorem:index}
Let \(X\) be a smooth and proper variety over \(k\).  Then for any
constructible complex \(\mathcal{F} \in D^b_c(X,\Lambda)\), we have
\[
\chi(X,\mathcal{F}) = \int_X \mathrm{cc}_{X}(\mathcal{F})
\]
where
\(\chi(X,\mathcal{F}) = \sum (-1)^i \dim \mathrm{H}^i(X;\mathcal{F})\)
is the Euler characteristic of \(\mathcal{F}\).
\end{theorem}

\begin{remark}
Over the field of complex numbers,
Schwartz~\cite{schwartz_characteristic-classes-1,schwartz_characteristic-classes-2}
and MacPherson~\cite{macpherson_chern-classes-for-singular-varieties}
introduced the Chern--Schwartz--MacPherson class for constructible
complexes, answering a conjecture of
Grothendieck~\cite[Note~87\textsubscript{1}]{grothendieck_recoltes-et-semailles}.
This gives a group homomorphism
\[
c_{X,\ast}(-) \colon K_0(D^b_c(X,\Lambda)) \to \mathrm{H}^{-\ast}(X,Ra_X^{!}\Lambda),
\]
where \(a_X \colon X \to \operatorname{Spec}\mathbb{C}\) is the
structural morphism.  Following Sabbah's conormal
construction~\cite{sabbah_remarks-conormal-spaces},
Kennedy~\cite{kennedy_macpherson-classes} gives an algebraic
construction in characteristic \(0\), and showed that the
Chern--Schwartz--MacPherson class factors through the Chow group.  In
the topological setting this was also done by
Ginzburg~\cite[Appendix]{ginsburg_characteristic-varieties-and-vanishing-cycles}.

Takeshi Saito's definition of characteristic classes in the étale
setting~\cite{saito_takeshi_characteristic-cycle-and-singular-support}
(cf.~\ref{situation:construction-cc}) follows the approach of
Ginzburg mentioned above.  In characteristic zero, the characteristic class
\(\mathrm{cc}_{X}(\mathcal{F})\) coincides with the ``shadow''
\(\check{c}_{X}(\mathcal{F})\) (a term coined by
Aluffi~\cite{aluffi_shadow}) of the characteristic cycle of
\(\mathcal{F}\), which differs from the Chern--Schwartz--MacPherson
class by a sign:
\[
\mathrm{cc}_{X,i}(\mathcal{F}) = \check{c}_{X,i}(\mathcal{F}) = (-1)^{i} c_{X,i}(\mathcal{F}).
\]
\end{remark}

\begin{remark}
In characteristic zero, the Chern--Schwartz--MacPherson class
gives rise to a natural transformation from \(K_0(D^b_c(X,\Lambda))\)
to \(\mathrm{CH}_{\ast}(X)\) for proper morphisms.  In
positive characteristic, however, the characteristic class
\(\mathrm{cc}_{X}(\mathcal{F})\) does not extend to a natural
transformation for all proper morphisms;
see~\cite[Example~6.10]{saito_takeshi_characteristic-cycle-and-singular-support}.

For historical remarks about the index formula,
see \cite[Chapter 9, Note]{kashiwara-schapira_sheaves-on-manifolds}.
\end{remark}

We shall need the following consequence of the index formula,
(in characteristic zero, see also \cite[Proposition~2.6]{aluffi_euler-characteristi-generali-linear-sections}).

\begin{lemma}\label{lemma:cc-hypersurface}
Let \(\mathbb{P}\) be a smooth projective variety with a very ample
invertible sheaf \(\mathscr{L}\).  Let
\(\mathcal{F} \in D^b_c(\mathbb{P},\Lambda)\).  Let
\(\iota\colon Y \to \mathbb{P}\) be the inclusion morphism of the
hypersurface \(Y\) defined by a sufficiently general section of
\(\mathscr{L}\).  Then we have
\[
\mathrm{cc}_{\mathbb{P}}(\iota_{\ast}\iota^{\ast}\mathcal{F}) = -\frac{c_1(\mathscr{L})}{1 - c_1(\mathscr{L})} \capprod \mathrm{cc}_{\mathbb{P}}(\mathcal{F}).
\]
In particular, by Theorem~\ref{theorem:index}, we get
\[
\chi(Y, \iota^{\ast}\mathcal{F}) = \int_{\mathbb{P}} -\frac{c_1(\mathscr{L})}{1 - c_1(\mathscr{L})} \capprod \mathrm{cc}_{\mathbb{P}}(\mathcal{F}).
\]
\end{lemma}

\begin{proof}
Because \(Y\) is a general hypersurface, by
\cite[Lemma~1.3.7]{saito_takeshi_characteristic-cycles-and-the-conductors-of-direct-image},
\(Y\) is properly
\(\mathrm{SS}(\mathcal{F})\)-transverse
in the sense of
\cite[Definition~7.1(1)]{saito_takeshi_characteristic-cycle-and-singular-support}.
Therefore, by \cite[Proposition~7.8,
Lemma~6.12]{saito_takeshi_characteristic-cycle-and-singular-support},
we have
\[
\mathrm{cc}_{Y}(\iota^{\ast}\mathcal{F}) = - c(T_Y^{\ast}\mathbb{P})^{-1} \capprod \iota^{!} \mathrm{cc}_{\mathbb{P}}(\mathcal{F}).
\]
Here, \(\iota^{!}\) is the Gysin pullback between the Chow groups
(cf.~\cite[\S2.6]{fulton_intersection-theory}).  Since \(\iota\) is a
closed immersion, by
\cite[Lemma~6.11]{saito_takeshi_characteristic-cycle-and-singular-support},
we have
\(\iota_{\ast}\mathrm{cc}_Y(\iota^{\ast}\mathcal{F}) = \mathrm{cc}_{\mathbb{P}}(\iota_{\ast}\iota^{\ast}\mathcal{F})\),
hence
\begin{align*}
  \mathrm{cc}_{\mathbb{P}}(\iota_{\ast}\iota^{\ast}\mathcal{F})
  &= \iota_{\ast}\left( - c(T_Y^{\ast}\mathbb{P})^{-1} \capprod \iota^{!} \mathrm{cc}_{\mathbb{P}}(\mathcal{F}) \right) \\
  &= \iota_{\ast}\left( {\textstyle - \frac{1}{1-c_1(\mathscr{L}|_Y)} \capprod \iota^{!} \mathrm{cc}_{\mathbb{P}}(\mathcal{F})} \right)
  & \text{by the adjunction formula}\\
  &= \iota_{\ast} \iota^{!}\left(\textstyle - \frac{1}{1-c_1(\mathscr{L})} \capprod \mathrm{cc}_{\mathbb{P}}(\mathcal{F})  \right)
  & \text{by \cite[Proposition 2.6(e)]{fulton_intersection-theory}} \\
  &= \textstyle -\frac{c_1(\mathscr{L})}{1-c_1(\mathscr{L})} \capprod \mathrm{cc}_{\mathbb{P}}(\mathcal{F})
  & \text{by \cite[Proposition 2.6(b)]{fulton_intersection-theory}}.
\end{align*}
This completes the proof.
\end{proof}

\begin{lemma}
\label{lemma:higher-intersection-cc}
Notation as in Lemma~\ref{lemma:cc-hypersurface}.
Suppose \(\iota\colon A \to \mathbb{P}\) is the intersection of \(r\) hypersurfaces
defined by \(r\) sufficiently general sections of \(\mathscr{L}\).
Then we have
\[
\mathrm{cc}_{\mathbb{P},j}(\iota_{\ast}\iota^{\ast}\mathcal{F})
=
(-1)^r \sum_{m\geq r} \binom{m-1}{r-1} \cdot c_1(\mathscr{L})^{m} \capprod \mathrm{cc}_{\mathbb{P},m+j}(\mathcal{F}).
\]
\end{lemma}

\begin{proof}
Using Lemma~\ref{lemma:cc-hypersurface} inductively,
we see
\[
\mathrm{cc}_{\mathbb{P}}(\iota_{\ast}\iota^{\ast}\mathcal{F})
= (-1)^r \left(\frac{c_1(\mathscr{L})}{1 - c_1(\mathscr{L})}\right)^r \capprod  \mathrm{cc}_{\mathbb{P}}(\mathcal{F}).
\]
Therefore, \(\mathrm{cc}_{\mathbb{P},j}(\iota_{\ast}\iota^{\ast}\mathcal{F})\)
equals \((-1)^r\sum_{m\geq j} a_m \cdot c_1(\mathscr{L})^{m} \capprod \mathrm{cc}_{\mathbb{P},m+j}\), where \(a_m\) is the coefficient of \(h^{m-r}\) of the power series
\(\left(\frac{1}{1 - h}\right)^r\), which is
\[
(-1)^{m-r} \binom{-r}{m-r} = \binom{r + m -r - 1}{m-r} = \binom{m-1}{r-1}.
\]
This completes the proof.
\end{proof}

\begin{definition}\label{definition:vir-gen-rank}
Let \(S\) be a variety over \(k\) of dimension \(n\), and let
\(\mathcal{G} \in D^{b}_{c}(S,\Lambda)\) be a constructible complex of
\(\Lambda\)-sheaves.  We define the \emph{virtual generic rank} of
\(\mathcal{G}\) to be
\[
\mathrm{rank}^{\circ}(\mathcal{G}) = \sum_{j\in \mathbb{Z}} (-1)^{j} \operatorname{rank}_{\Lambda} \mathcal{H}^{j}(\mathcal{G}|_{S^{\circ}})
\]
where \(S^{\circ}\) is the largest purely \(n\)-dimensional Zariski
open dense subset of \(S\) over which \(\mathcal{G}\) is lisse.
\end{definition}

\begin{lemma}\label{lemma:range-of-cc}
Let \(X\) be a smooth variety over \(k\).  Let
\(\mathcal{F} \in D^{b}_{c}(X,\Lambda)\) be a constructible complex.
Let \(S\) be the support of \(\mathcal{F}\) endowed with the reduced
structure.  Let \(n = \dim S\).  Then we have
\begin{enumerate}
\item \(\mathrm{cc}_{X,j}(\mathcal{F}) = 0\) for \(j > n\),
\item
\(\mathrm{cc}_{X,n}(\mathcal{F}) = (-1)^{n} \cdot \mathrm{rank}^{\circ}(\mathcal{F}) \cdot [S]_n \in \mathrm{CH}_{n}(X)\),
where \([S]_n\) is the fundamental class of the
union of \(n\)-dimensional irreducible components of \(S\).
\end{enumerate}
\end{lemma}

\begin{proof}[Proof (cf.~{\cite[Lemma 6.9(2)]{saito_takeshi_characteristic-cycle-and-singular-support}})]
The first part is clear since \(\mathrm{CH}_j(S) = 0\) for \(j > 0\).
Write \(S = S^{\prime} \cup \bigcup_{i=1}^m S_i\), where \(\dim S^{\prime} \leq n-1\)
and each \(S_i\) is an irreducible component of \(S\) of dimension \(n\).
Thus, \(\mathrm{CH}_n(S) = \bigoplus_{i=1}^m \mathrm{CH}_n(S_i)\).
Let \(S^{\circ}\) be the largest open subset of \(\bigcup_{i} S_i\)
over which \(\mathcal{F}\) is lisse.
Then, \(\mathrm{CC}(\mathcal{F}|_{S^{\circ}})\) is \((-1)^n \operatorname{rank}^{\circ}(\mathcal{F}|_{S^{\circ}})\)
times the zero section of the cotangent bundle of \(S^{\circ}\).
The assertion follows (see \cite[Definition~6.7]{saito_takeshi_characteristic-cycle-and-singular-support}).
\end{proof}

\renewcommand{\theremark}{\thesection.\arabic{remark}}

\section{The proper degeneration lemma}
\label{sec:proper-degeneration}

In this section, we prove the following Proper Degeneration
Lemma, which is a major technical engine of this paper.

\begin{lemma}[Proper Degeneration Lemma]%
\label{lemma:proper-degeneration}
Let \(\mathbb{P}\) be a proper variety over \(k\).  Let
\(\mathscr{L}\) be an invertible sheaf on \(\mathbb{P}\).  Let
\(\sigma_{0}, \sigma_{\infty}\) be two linearly independent elements
of \(\mathrm{H}^{0}(\mathbb{P};\mathscr{L})\).  For
\(t=[u,v] \in \mathbb{P}^{1}\), let \(Y_t \subset \mathbb{P}\) be the divisor
defined by the section \(u\sigma_{0}-v\sigma_{\infty}\).  Let
\(B = Y_{0} \cap Y_{\infty}\).  Let \(\mathcal{F}\) be an object of
\(D^b_c(\mathbb{P},\Lambda)\).
\begin{itemize}
\item For a general \(t \in \mathbb{P}^{1}\), and
any \(m\), we have
\begin{align}\label{eq:estimate-ineq}
  \dim  \mathrm{H}^{m}(Y_0;\mathcal{F}) \leq
  \ & \dim \mathrm{H}^{m}(Y_t;\mathcal{F}) +
      \dim \mathrm{H}^{m-1}(Y_t;\mathcal{F}(-1)) + \dim \mathrm{H}^{m}(\mathbb{P};\mathcal{F}) \\
    &+ \dim \operatorname{Coker}(\mathrm{H}^{m-2}(Y_{t};\mathcal{F}(-1)) \to \mathrm{H}^{m-2}(B;\mathcal{F}(-1))) \nonumber
  \\ \nonumber
    &+ \dim \operatorname{Ker}(\mathrm{H}^{m-1}(Y_{t};\mathcal{F}(-1)) \to \mathrm{H}^{m-1}(B;\mathcal{F}(-1)))\\\nonumber
  \leq\ & \dim \mathrm{H}^{m}(Y_{t};\mathcal{F})
          + 2\dim \mathrm{H}^{m-1}(Y_{t};\mathcal{F}) \\
         & + \dim \mathrm{H}^{m}(\mathbb{P};\mathcal{F})
          + \dim \mathrm{H}^{m-2}(B;\mathcal{F}). \nonumber
\end{align}
\item
If we assume additionally that
\(\mathcal{F} \in {}^{\mathrm{p}}D_{c}^{\geq0}(\mathbb{P},\Lambda)\), and
\(\mathscr{L}\) is ample, then for \(m \leq -1\), and
\(t \in \mathbb{P}^{1}\) sufficiently general, we have
\begin{equation}\label{eq:proposition}
\dim \mathrm{H}^{m}(Y_{0};\mathcal{F}) \leq
\dim \mathrm{H}^{m}(Y_{t};\mathcal{F}) +
\dim \mathrm{H}^{m-1}(Y_{t};\mathcal{F}(-1)) +
\dim \mathrm{H}^{m}(\mathbb{P};\mathcal{F}).
\end{equation}
\end{itemize}
\end{lemma}

\begin{remark}
The inequality \eqref{eq:proposition} can be viewed as a proper
version of the perverse degeneration lemma for affine morphisms (see
\cite[Proposition~9]{katz_perverse-origin},
\cite[Lemma~3.1]{wan-zhang_betti-number-bounds-for-varieties-and-exponential-sums0}).

For the study of middle-dimensional cohomology, the perverse
degeneration lemma suffices; for example,
Theorem~\ref{theorem:main}(2) can be derived from it.  However, since
we consider possibly highly singular varieties, it is necessary to
control cohomology groups beyond the middle dimension.  In these
situations, standard arguments such as the weak Lefschetz theorem are
not sufficient.  Instead, the more general (though weaker) estimate
given by~\eqref{eq:estimate-ineq} becomes essential.

Unlike the perverse degeneration lemma, where the total space of the
degeneration plays no role, the approach of
Lemma~\ref{lemma:proper-degeneration} links the cohomology of the
special (degenerating) hypersurface to both the cohomology of a
general fiber and that of the ambient variety.
\end{remark}

Before proving the Proper Degeneration Lemma,
we first establish a few lemmas that will be used in making this
connection.

\begin{lemma}%
\label{lemma:local-situation-estimate}
Let \((S,s,\eta,\overline{\eta})\) be the strict henselian trait with
residue field \(k\), generic point \(\eta\), and a fixed geometric
generic point \(\overline{\eta}\).  Let \(\mathcal{Y} \to S\) be a
proper \(S\)-scheme.  Let \(Y_s\), \(Y_{\eta}\), and
\(Y_{\overline{\eta}}\) be the special fiber, generic fiber, and
geometric generic fiber of \(\mathcal{Y}\) respectively. Then for any
\(\mathcal{F} \in D^b_c(\mathcal{Y},\Lambda)\), we have
\begin{align*}
\dim \mathrm{H}^m(Y_s;\mathcal{F}) &\leq
\dim \mathrm{H}^m(Y_{\overline{\eta}};\mathcal{F}) + \dim \mathrm{H}^{m-1}(Y_{\overline{\eta}};\mathcal{F}(-1)) \\
 &\quad  + \dim \operatorname{Im}(\mathrm{H}^m_{Y_s}(\mathcal{Y};\mathcal{F}) \to \mathrm{H}^m(Y_s;\mathcal{F})).
\end{align*}
\end{lemma}

\begin{proof}
By the proper base change theorem, we have
\(\mathrm{H}^{m}(\mathcal{Y};\mathcal{F}) \simeq \mathrm{H}^{m}(Y_s;\mathcal{F})\).  The
relative exact sequence for the inclusions
\(Y_s \xrightarrow{i} \mathcal{Y} \xleftarrow{j} Y_{\eta}\) reads:
\begin{equation*}
\cdots \to  \mathrm{H}^{m}_{Y_s}(\mathcal{Y};\mathcal{F}) \to \mathrm{H}^{m}(Y_s;\mathcal{F}) \to
\mathrm{H}^{m}(Y_{\eta};\mathcal{F}) \to \cdots.
\end{equation*}
Thus
\begin{equation}
\label{eq:relative-local-situation}
\dim \mathrm{H}^{m}(Y_{s};\mathcal{F}) \leq
\dim \operatorname{Im}(\mathrm{H}^{m}_{Y_{s}}(\mathcal{Y};\mathcal{F})\to\mathrm{H}^{m}(Y_{s};\mathcal{F}))
+ \dim \mathrm{H}^{m}(Y_{\eta};\mathcal{F}).
\end{equation}

Let \(R\Psi_{\overline{\eta}}(-)\) denote the nearby cycle functor
with respect to \(\overline{\eta}\).  Then we have the variation
triangle
\begin{equation*}
i^{\ast}Rj_{\ast}\mathcal{F} \to R\Psi_{\overline{\eta}}(\mathcal{F}) \to
R\Psi_{\overline{\eta}}(\mathcal{F})(-1),
\end{equation*}
which, after taking \(R\Gamma\) and applying the proper base change
theorem, gives rise to a long exact sequence, known as \emph{Hsien
Chung Wang's sequence}
\begin{equation}\label{eq:wang-sequence}
\cdots \to \mathrm{H}^{m-1}(Y_{\overline{\eta}};\mathcal{F}(-1)) \to \mathrm{H}^{m}(Y_{\eta};\mathcal{F})
\to \mathrm{H}^{m}(Y_{\overline{\eta}};\mathcal{F}) \to \cdots,
\end{equation}
By combining~\eqref{eq:relative-local-situation} and the Wang
sequence~\eqref{eq:wang-sequence}, we complete the proof of the Lemma.
\end{proof}

\begin{lemma}
\label{lemma:use-global-geometry-to-estimate-cohomology-with-support}
Let \(C\) be a smooth irreducible curve over \(k\) with a point
\(0 \in C\).  Let \(f\colon \widetilde{\mathcal{Y}} \to C\) be a proper
morphism.  For a point \(t\) of \(C\), let \(Y_t = f^{-1}(t)\).  Then
for any \(\mathcal{F} \in D^b_c(\widetilde{\mathcal{Y}},\Lambda)\) any
sufficiently general point \(t\) of \(C\), we have
\begin{equation*}
\dim \mathrm{H}^m(Y_0;\mathcal{F}) \leq
\dim \mathrm{H}^m(Y_t;\mathcal{F}) + \dim \mathrm{H}^{m-1}(Y_t;\mathcal{F}(-1))
+ \dim \operatorname{Im}(\mathrm{H}^m(\widetilde{\mathcal{Y}};\mathcal{F}) \to \mathrm{H}^m(Y_0;\mathcal{F})).
\end{equation*}
\end{lemma}

\begin{proof}
Let \(S\) be the strict localization of \(C\) at \(0\).  In the
context of Lemma~\ref{lemma:local-situation-estimate}, we have
\(s = 0\).  Let \(\mathcal{Y} = \widetilde{\mathcal{Y}}\times_{C} S\).  Since
\(R^{m}f_{\ast}\mathcal{F}\) is a constructible sheaf on \(C\), there
exists a Zariski open dense subset where the fiber dimension of
\(R^{m}f_{\ast}\mathcal{F}\) is constant.  By the proper base change
theorem, the fiber of \(R^{m}f_{\ast}\mathcal{F}\) at a point \(t\) is
identified with \(\mathrm{H}^m(Y_{t};\mathcal{F})\) and for these \(t\), we have
\(\dim\mathrm{H}^m(Y_t;\mathcal{F})=\dim\mathrm{H}^m(Y_{\overline{\eta}};\mathcal{F})\).
We have the following commutative diagram
\begin{equation*}
\begin{tikzcd}
\mathrm{H}^m_{Y_{0}}(\widetilde{\mathcal{Y}};\mathcal{F}) \ar[r,"\sim"] \ar[d] & \mathrm{H}^m_{Y_0}(\mathcal{Y};\mathcal{F}) \ar[d] \\
\mathrm{H}^m(\widetilde{\mathcal{Y}};\mathcal{F}) \ar[r] & \mathrm{H}^m(\mathcal{Y};\mathcal{F}) \ar[r,equal] & \mathrm{H}^m(Y_{0};\mathcal{F}).
\end{tikzcd}
\end{equation*}
Since local cohomology only depends on a neighborhood of \(Y_0\),
the first horizontal arrow is an isomorphism.  By the commutativity
of the diagram, we conclude that
\(\operatorname{Im}(\mathrm{H}^m_{Y_0}(\mathcal{Y};\mathcal{F}) \to \mathrm{H}^m(\mathcal{Y};\mathcal{F}))\) is
contained in
\(\operatorname{Im}(\mathrm{H}^m(\widetilde{\mathcal{Y}};\mathcal{F})\to\mathrm{H}^m(\mathcal{Y};\mathcal{F}))\).
Lemma~\ref{lemma:use-global-geometry-to-estimate-cohomology-with-support}
then follows from Lemma~\ref{lemma:local-situation-estimate}.
\end{proof}

\begin{proof}[Proof of Proper Degeneration Lemma~\ref{lemma:proper-degeneration}]
To simplify notation in the proof, all functors between derived
categories are interpreted as their derived versions.  For example,
\(\beta_{\ast}\) denotes \(R\beta_{\ast}\), and similarly for other
functors.

Let \(\mathcal{X}\) be the divisor of \(\mathbb{P} \times \mathbb{P}^{1}\) cut
out by \(u\sigma_{0} - v\sigma_{\infty}=0\).  Let
\(\varpi\colon \mathcal{X} \to \mathbb{P}\) be the restriction of
\(\mathrm{pr}_{1}\colon \mathbb{P}\times \mathbb{P}^{1} \to \mathbb{P}\).  Let
\(f\colon \mathcal{X} \to \mathbb{P}^{1}\) be the restriction of
\(\mathrm{pr}_{2}\colon \mathbb{P} \times \mathbb{P}^{1} \to \mathbb{P}^{1}\).
Define \(\mathcal{X}^{\circ}=f^{-1}(\mathbb{A}^{1}) \subset \mathcal{X}\).
Then we have \(\varpi^{-1}(B) = E \simeq B \times \mathbb{P}^{1}\),
and \(\varpi^{-1}(B) \cap \mathcal{X}^{\circ} = E^{\circ} \simeq B \times \mathbb{A}^{1}\).

In view of
Lemma~\ref{lemma:use-global-geometry-to-estimate-cohomology-with-support}
(with \(\widetilde{\mathcal{Y}} = \mathcal{X}^{\circ}\)),
in order to prove the first assertion,
it suffices to show that for any \(m\),
\begin{align}\label{eq:bounding-global-betti}
  \dim \mathrm{H}^{m}(\mathcal{X}^{\circ};\varpi^{\ast}\mathcal{F})
  \leq\ & \dim  \mathrm{H}^{m}(\mathbb{P};\mathcal{F})
             + \dim \operatorname{Coker}(\mathrm{H}^{m-2}(Y_{t};\mathcal{F}(-1)) \to \mathrm{H}^{m-2}(B;\mathcal{F}(-1))) \\
           & +\dim  \operatorname{Ker}(\mathrm{H}^{m-1}(Y_{t};\mathcal{F}(-1)) \to \mathrm{H}^{m-1}(B;\mathcal{F}(-1))). \nonumber
\end{align}
Consider the following commutative diagram, in which are squares are
cartesian:
\begin{equation*}
\begin{tikzcd}
E^{\circ} \ar[r,"i^{\circ}"] \ar[d,"\varpi|_{E^{\circ}}"] & \mathcal{X}^{\circ} \ar[d,"\varpi^{\circ}"] & \mathcal{X}^{\circ}\setminus E^{\circ} \ar[l,swap,"j^{\circ}"] \ar[d,"\varpi|_{\mathcal{X}^{\circ}\setminus E^{\circ}}"]\\
B \ar[r,swap,"\beta"] & \mathbb{P} & \mathbb{P}\setminus B \ar[l,"\alpha"]
\end{tikzcd}.
\end{equation*}
In the diagram, we have used the notation \(\varpi^{\circ}=\varpi|_{\mathcal{X}^{\circ}}\).
Note that
\(\varpi^{\circ}_{\ast}(\varpi^{\ast}\mathcal{F})|_{\mathcal{X}^{\circ}}=\varpi^{\circ}_{\ast}\varpi^{\circ\ast}\mathcal{F}\).
Thus we get the following (solid)
commutative diagram, in which each row is a distinguished triangle
\begin{equation}
\label{eq:diagram-compare}
\begin{tikzcd}[column sep=small]
\mathcolor{blue}{\alpha_{!}\alpha^{\ast}\mathcal{F}} \ar[r] \ar[d]& \mathcal{F} \ar[d] \ar[r] & \beta_{\ast}\mathcal{F}|_{B} \ar[d]\\
\mathcolor{blue}{\varpi_{\ast}^{\circ}j_{!}^{\circ}j^{\circ\ast}\varpi^{\circ\ast}\mathcal{F}} \ar[r] \ar[d,dashed]
& \varpi^{\circ}_{\ast}\varpi^{\circ\ast}\mathcal{F} \ar[r] \ar[d,dashed]&
\varpi^{\circ}_{\ast}i^{\circ}_{\ast}(\varpi^{\circ\ast}\mathcal{F})|_{E^{\circ}} \ar[r,equal]  & \beta_{\ast}(\varpi|_{E^{\circ}})_{\ast}(\varpi|_{E^{\circ}})^{\ast}(\mathcal{F}|_{B}) \\
\mathcal{G}^{\prime} \ar[r,equal,dashed] & \mathcal{G}^{\prime}
\end{tikzcd}
\end{equation}
Since \(E^{\circ}=B \times \mathbb{A}^{1}\),
the right vertical arrow is an isomorphism.  It follows
that the cones of the left and middle vertical arrows are isomorphic.
We denote these cones by \(\mathcal{G}^{\prime}\).

To prove \eqref{eq:bounding-global-betti}, we consider the following
commutative diagram
\begin{equation*}
\begin{tikzcd}[column sep=small]
\mathrm{H}^m_{\mathrm{c}}(\mathbb{P}\setminus B;\mathcal{F}) \ar[r]\ar[d] & \mathrm{H}^m(\mathbb{P};\mathcal{F}) \ar[r] \ar[d] &
\mathrm{H}^m(B;\mathcal{F}) \ar[d]  \ar[r] &\mathrm{H}^{m+1}_{\mathrm{c}}(\mathbb{P}\setminus B;\mathcal{F}) \ar[d]  \\
\mathrm{H}^m(\mathcal{X}^{\circ},E^{\circ};\varpi^{\ast}\mathcal{F}) \ar[r] & \mathrm{H}^m(\mathcal{X}^{\circ};\varpi^{\ast}\mathcal{F})  \ar[r]
& \mathrm{H}^m(E^{\circ};\varpi^{\ast}\mathcal{F})  \ar[r] &  \mathrm{H}^{m+1}(\mathcal{X}^{\circ},E^{\circ};\varpi^{\ast}\mathcal{F})
\end{tikzcd}
\end{equation*}
induced by the relative cohomology sequences.  In the diagram, all
rows are exact.  By the discussion from the above paragraph, we see
\[
\dim \mathrm{H}^{m}(\mathcal{X}^{\circ};\varpi^{\ast}\mathcal{F})
\leq \dim \mathrm{H}^{m}(\mathbb{P};\mathcal{F})+
\dim \mathrm{H}^{m}(\mathbb{P};\mathcal{G}^{\prime}),
\]
where \(\mathcal{G}^{\prime}\) is the cofiber (or cone) of the left
vertical arrow in the diagram \eqref{eq:diagram-compare}.

To complete the proof, we consider the following diagram
\begin{equation*}
\begin{tikzcd}
B \ar[r,"i_{\infty}"] \ar[d,"a"] & Y_{\infty} \ar[d,"u"] & Y_{\infty} \setminus B \ar[l,swap,"j_{\infty}"] \ar[d]\\
E \ar[r,"i"] & \mathcal{X} & \mathcal{X} \setminus E \ar[l,swap,"j"] \\
E^{\circ}\ar[r,"i^{\circ}"] \ar[u] & \mathcal{X}^{\circ}\ar[u,"v"]
& \mathcal{X}^{\circ}\setminus E^{\circ} \ar[l,swap,"j^{\circ}"]\ar[u]
\end{tikzcd}.
\end{equation*}
This diagram induces a diagram of distinguished triangles
\begin{equation*}
\begin{tikzcd}[column sep=small]
u_{\ast}u^{!} j_{!} j^{\ast}\varpi^{\ast}\mathcal{F} \ar[r] \ar[d] & u_{\ast}u^{!}\varpi^{\ast}\mathcal{F} \ar[r] \ar[d] &
u_{\ast}u^{!} i_{\ast}i^{\ast}\varpi^{\ast}\mathcal{F} \ar[d] \ar[r,equal] & u_{\ast}i_{\infty\ast}a^{!}i^{\ast}\varpi^{\ast}\mathcal{F}\\
j_{!}j^{\ast}\varpi^{\ast}\mathcal{F} \ar[r] \ar[d] & \varpi^{\ast}\mathcal{F} \ar[r] \ar[d] &
i_{\ast}i^{\ast}\varpi^{\ast}\mathcal{F} \ar[d]\\
v_{\ast}v^{\ast}j_{!}j^{\ast}\varpi^{\ast}\mathcal{F}\ar[r] \ar[d,equal]&
v_{\ast}\varpi^{\ast}\mathcal{F} \ar[r] &  v_{\ast}v^{\ast}i_{\ast}i^{\ast}\varpi^{\ast}\mathcal{F}\ar[d,equal]\\
v_{\ast}j^{\circ}_{!} (j^{\ast}\varpi^{\ast}\mathcal{F})|_{\mathcal{X}^{\circ}\setminus E^{\circ}}  & {}
& v_{\ast} i^{\circ}_{\ast} (i^{\ast}\varpi^{\ast}\mathcal{F})|_{E^{\circ}}.
\end{tikzcd}
\end{equation*}
The equal signs in the diagram are valid due to the proper base change theorem.
Applying the global section functor
\(\Gamma(\mathcal{X},\cdot)\) gives the following diagram, whose
rows and columns are all exact:
\begin{equation*}
\begin{tikzcd}[column sep=7pt]
\mathrm{H}^{m-1}_{Y_{\infty}}(\mathcal{X};\varpi^{\ast}\mathcal{F}) \ar[r,red] \ar[d]& \mathrm{H}^{m-1}_B(E;\varpi^{\ast}\mathcal{F}) \ar[r] \ar[d]
& \mathrm{H}^m(Y_{\infty};\mathcal{G}) \ar[r] \ar[d] &
\mathrm{H}^m_{Y_{\infty}}(\mathcal{X};\varpi^{\ast}\mathcal{F}) \ar[r,red] \ar[d] & \mathrm{H}^m_B(E;\varpi^{\ast}\mathcal{F}) \ar[d]\\
\mathrm{H}^{m-1}(\mathcal{X};\varpi^{\ast}\mathcal{F}) \ar[r] \ar[d] & \mathrm{H}^{m-1}(E;\varpi^{\ast}\mathcal{F}) \ar[d] \ar[r]
&  \mathcolor{blue}{\mathrm{H}^m_{\mathrm{c}}(\mathcal{X}\setminus E;\varpi^{\ast}\mathcal{F})} \ar[r] \ar[d] &
\mathrm{H}^m(\mathcal{X};\varpi^{\ast}\mathcal{F}) \ar[d] \ar[r] & \mathrm{H}^m(E;\varpi^{\ast}\mathcal{F}) \ar[d]\\
\mathrm{H}^{m-1}(\mathcal{X}^{\circ};\varpi^{\ast}\mathcal{F}) \ar[r] \ar[d] & \mathrm{H}^{m-1}(E^{\circ};\varpi^{\ast}\mathcal{F}) \ar[r] \ar[d]
& \mathcolor{blue}{\mathrm{H}^m(\mathcal{X}^{\circ},E^{\circ};\varpi^{\ast}\mathcal{F})} \ar[r] \ar[d]&
\mathrm{H}^m(\mathcal{X}^{\circ};\varpi^{\ast}\mathcal{F}) \ar[r] \ar[d] & \mathrm{H}^{m}(E^{\circ};\varpi^{\ast}\mathcal{F}) \ar[d]\\
\mathrm{H}^m_{Y_{\infty}}(\mathcal{X};\varpi^{\ast}\mathcal{F}) \ar[r,red] &
\mathrm{H}^m_B(E;\varpi^{\ast}\mathcal{F}) \ar[r] & \mathrm{H}^{m+1}(Y_{\infty};\mathcal{G}) \ar[r] &
\mathrm{H}^{m+1}_{Y_{\infty}}(\mathcal{X};\varpi^{\ast}\mathcal{F}) \ar[r,red] & \mathrm{H}^{m+1}_B(E;\varpi^{\ast}\mathcal{F})
\end{tikzcd}
\end{equation*}
where \(\mathcal{G}=u^{!}j_{!}j^{\ast}\varpi^{\ast}\mathcal{F}\).

By construction, we have \(\mathbb{P}\setminus B = \mathcal{X}\setminus E\),
whence (compare with the blue items in \eqref{eq:diagram-compare})
\begin{equation*}
\mathrm{H}^m_{\mathrm{c}}(\mathbb{P}\setminus B;\mathcal{F}) =
\mathrm{H}^m_{\mathrm{c}}(\mathcal{X}\setminus E;\varpi^{\ast}\mathcal{F}).
\end{equation*}
It follows from the five lemma that
\(\mathrm{H}^{m}(\mathbb{P};\mathcal{G}^{\prime})\cong\mathrm{H}^{m+1}(Y_{\infty};\mathcal{G})\).

It remains to show that
\begin{align}\label{eq:key-inequality}
  \dim \mathrm{H}^{m+1}(Y_{\infty};\mathcal{G}) \leq \
  &\dim \operatorname{Coker}(\mathrm{H}^{m-2}(Y_{t};\mathcal{F}(-1)) \to \mathrm{H}^{m-2}(B;\mathcal{F}(-1))) \\
  &+ \dim \operatorname{Ker}(\mathrm{H}^{m-1}(Y_{t};\mathcal{F}(-1)) \to \mathrm{H}^{m-1}(B;\mathcal{F}(-1))).\nonumber
\end{align}

The inequality~\eqref{eq:key-inequality} does not always hold in
general.  But under certain additional assumptions on \(Y_{\infty}\),
which we are always able to arrange by performing a linear change of
coordinates on \(\mathbb{P}^1\), the
inequality~\eqref{eq:key-inequality} holds.  To state these additional
hypotheses, let us consider the morphism
\(f\colon \mathcal{X} \to \mathbb{P}^1\).  We denote
\(L = f_{\ast}\varpi^{\ast}\mathcal{F}\) and
\(L^{\prime} = f|_{E\ast}(\varpi^{\ast}\mathcal{F}|_{E})\).  According
to the relative purity theorem (see \cite[Exposé XVI, 3.3]{sga4} or
\cite[Theorem~11.2,~Supplement]{kiehl-weissauer_weil-conjectures-perverse-sheaves}),
for any \(t\) in the Zariski open dense subset \(U\) of
\(\mathbb{P}^1\) where the cohomology sheaves of \(L\) and
\(L^{\prime}\) are all lisse, we have
\(\iota_t^{!}L \simeq \iota_{t}^{\ast}L[-2](-1)\) and
\(\iota_t^{!}L^{\prime} \simeq \iota_t^{\ast}L^{\prime}[-2](-1)\),
where \(\iota_t\) is the inclusion morphism
\(\{t\} \to \mathbb{P}^1\).  Our hypothesis is that \(\infty\) lies in
this ``ouvert de lissité'' \(U\).  This can always be achieved by
relabeling the hypersurfaces \(Y_t\) without changing \(Y_0\).  Under
this assumption, we have an isomorphism
\(\mathrm{H}^{m}(Y_{\infty};\mathcal{F})\cong\mathrm{H}^{m}(Y_{t};\mathcal{F})\)
for any \(t \in U\), as \(\Lambda\)-modules.

Since we have arranged that
\(\iota_{\infty}^{!}f_{\ast}\varpi^{\ast}\mathcal{F}\simeq\iota_{\infty}^{\ast}f_{\ast}\varpi^{\ast}\mathcal{F}[-2](-1)\),
applying proper base change to the following diagram
\begin{equation*}
\begin{tikzcd}
Y_{\infty}\ar[r,"u"] \ar[d] & \mathcal{X} \ar[d,"f"] \\ \{\infty\}\ar[r,"\iota_{\infty}"] & \mathbb{P}^1
\end{tikzcd},
\end{equation*}
we have
\begin{equation*}
\begin{tikzcd}[column sep=small]
\iota_{\infty}^{\ast}f_{\ast}\varpi^{\ast}\mathcal{F}[-2](-1) \ar[d,"\sim" {anchor=south, rotate=90}] \ar[r,equal] &
\Gamma(Y_{\infty};u^{\ast}\varpi^{\ast}\mathcal{F}[-2](-1)) \ar[r,equal] & \Gamma(Y_{\infty};\mathcal{F}[-2](-1)) \\
\iota_{\infty}^{!}f_{\ast}\varpi^{\ast}\mathcal{F} \ar[r,equal] &
\Gamma(Y_{\infty};u^{!}\varpi^{\ast}\mathcal{F})
\end{tikzcd}
\end{equation*}
Since \(E = \mathbb{P}^{1} \times B\), the relative purity isomorphism gives
\begin{equation*}
\Gamma(B;\mathcal{F}[-2](-1)) \simeq \Gamma_{B}(E;\varpi^{\ast}\mathcal{F}).
\end{equation*}
It follows that
\(\mathrm{H}^{\ast}_{Y_{\infty}}(\mathcal{X};\varpi^{\ast}\mathcal{F})\to\mathrm{H}_B^{\ast}(E;\varpi^{\ast}\mathcal{F})\)
can be identified with the natural restriction map
\[\mathrm{H}^{\ast-2}(Y_{\infty};\mathcal{F}(-1)) \to \mathrm{H}^{\ast-2}(B;\mathcal{F}(-1)).\]
We have finished the proof of the first part of the theorem.

Next we prove the second part.  We thus assume that
\(\mathcal{F} \in {}^{\mathrm{p}}D_{c}^{\geq 0}({\mathbb{P}},\Lambda)\), and
\(\mathscr{L}\) is ample.  Since \(\mathbb{P} \times \mathbb{P}^1\) is
smooth over \(\mathbb{P}\), we know that
\(\mathrm{pr}_1^{\ast}\mathcal{F}[1]\) is in
\({}^{\mathrm{p}}D_c^{\geq0}(\mathbb{P}\times \mathbb{P}^1,\Lambda)\).

Since \(\mathcal{X}\) is locally defined by a single
equation in \(\mathbb{P} \times \mathbb{P}^1\),
Lemma~\ref{lemma:estimate-perv-dergee} implies that
\(\varpi^{\ast}\mathcal{F}=(\mathrm{pr}_1^{\ast}\mathcal{F}[1])|_{\mathcal{X}}[-1]\)
is perverse coconnective.  As \(Y_{t}\) is a Cartier divisor
of \(\mathbb{P}\), by the same lemma, \(\mathcal{F}[-1]|_{Y_{t}}\)
is also perverse coconnective.

Since \(\mathscr{L}\) is an ample line bundle, \(Y_{t} \setminus B\)
is an affine variety.  By applying Artin's vanishing theorem
(Theorem~\ref{theorem:relative-affine-lefschetz}) to
\(\mathcal{F}[-1]|_{Y_{\infty}}\), we conclude that
\(\mathrm{H}^i_c(Y_{t} \setminus B;\mathcal{F}[-1]) = 0\) for
\(i < 0\).  This implies that diagram
\begin{equation*}
\mathrm{H}^{m-1}(Y_{t};\mathcal{F}[-1]) \to \mathrm{H}^{m-1}(B;\mathcal{F}[-1])
\end{equation*}
is bijective for \(m<0\), and injective for \(m=0\).  This shows that
in the range \(m\leq 0\), the ``kernel'' summand and the ``cokernel''
summand in \eqref{eq:estimate-ineq} are zero.  The proposition is
proved.
\end{proof}

\section{Hypersurface restriction of perverse sheaves }
\label{sec:main-theorem}
\begin{situation}[Setup]\label{setup}
Throughout this section, let \(\mathbb{P}\) be a smooth projective
variety over \(k\) endowed with a very ample invertible sheaf
\(\mathscr{L}\).  For simplicity, we shall refer to the hypersurface
cut out by a section of \(\mathscr{L}^{\otimes d}\) as a ``degree
\(d\) hypersurface'' in \(\mathbb{P}\).  If
\(\mathcal{F} \in D^b_c(\mathbb{P},\Lambda)\), we write
\(\mathrm{cc}(\mathcal{F}) \in \mathrm{CH}_{\ast}(\mathbb{P})\)
instead of \(\mathrm{cc}_{\mathbb{P}}(\mathcal{F})\) and \(\mathrm{cc}_j(\mathcal{F})\) for \(\mathrm{cc}_{\mathbb{P},j}(\mathcal{F})\) where \(j \in \mathbb{Z}_{\geq0}\).
\end{situation}

\begin{definition}
Let \(\mathcal{F} \in D^{b}_{c}(\mathbb{P},\Lambda)\) be a
constructible complex of \(\Lambda\)-sheaves with support
\(S = \operatorname{Supp}(\mathcal{F})\), endowed with the reduced
structure.  We associate three numbers to \(\mathcal{F}\)
\begin{itemize}
\item We define \(n(\mathcal{F})\) to be the dimension of \(S\).
\item Let \(S^{\prime}\) be the union of irreducible components of
\(S\) of dimension \(n(\mathcal{F})\).  We define
\(\delta(\mathcal{F}) = \deg S^{\prime}\) to be the degree of
\(S^{\prime}\) with respect to \(\mathscr{L}\).
\item We define
\(r(\mathcal{F}) = (-1)^{n(\mathcal{F})} \cdot \operatorname{rank}^{\circ}(\mathcal{F}|_S)\)
to be the virtual generic rank of \(\mathcal{F}|_{S}\) times
\((-1)^{n(\mathcal{F})}\).
(Definition~\ref{definition:vir-gen-rank}).  Note that if
\(\mathcal{F}\) is a nonzero perverse sheaf, then
\(r(\mathcal{F}) \geq 1\).
\end{itemize}
\end{definition}

The goal of this section is to give a proof of the following theorem.

\begin{theorem}\label{theorem:main}
Let notation be as in \ref{setup}.  Suppose \(\mathcal{P}\) is a
perverse sheaf on \(\mathbb{P}\).  Then there exist constants
\(C_1, C_2, C_3 > 0\), independent of the choice of \(Y\) below, such
that
\begin{enumerate}
\item For any degree \(d\) hypersurface \(Y \subset \mathbb{P}\),
we have
\[
\dim \mathrm{H}^{-1}(Y;\mathcal{P}|_Y) \leq r(\mathcal{P}) \cdot \delta(\mathcal{P}) \cdot d^{n(\mathcal{P})} +
C_1 \cdot d^{n(\mathcal{P})-1}.
\]
\item For any degree \(d\) hypersurface \(Y \subset \mathbb{P}\) such
that \(\mathcal{P}|_Y[-1]\) is a perverse sheaf, we have
\[
B(Y, \mathcal{P}|_{Y}) \leq r(\mathcal{P}) \cdot \delta(\mathcal{P}) \cdot d^{n(\mathcal{P})} + C_2 \cdot d^{n(\mathcal{P})-1}.
\]
\item For any degree \(d\) hypersurface \(Y \subset \mathbb{P}\) whatsoever, we
have
\[
B(Y, \mathcal{P}|_{Y}) \leq 3\cdot r(\mathcal{P}) \cdot \delta(\mathcal{P}) \cdot d^{n(\mathcal{P})} + C_3 \cdot d^{n(\mathcal{P})-1}.
\]
\end{enumerate}
\end{theorem}

\begin{remark}
As we will see, the constant \(C_1\) can be made quite explicit
(see~\ref{situation:explication-1}), as it is only related to the
characteristic classes \(\mathrm{cc}_j(\mathcal{P})\) of
\(\mathcal{P}\) where \(j < n(\mathcal{P})\), as well as the Betti
numbers \(\dim \mathrm{H}^{-1}(\mathbb{P},\mathcal{P})\) and
\(\dim \mathrm{H}^{-2}(\mathbb{P};\mathcal{P})\).

The constants \(C_2, C_{3}\) are slightly more complicated than
\(C_{1}\) (see~\ref{situation:explication-2},
\ref{situation:explication-3}), but they are ultimately determined by
characteristic numbers associated to
\(\mathrm{cc}(\mathcal{P}|_{A_j})\), where
\(A_0 \supset A_1 \supset \cdots\) is a general flag of linear
subspaces in \(\mathbb{P}\).  Here, a ``linear subspace'' refers to an
intersection of divisors defined by nonzero sections in
\(\mathrm{H}^0(\mathbb{P};\mathscr{L})\).

The factor \(3\) appearing in Theorem~\ref{theorem:main}(3) reflects a
technical issue: we are currently unable to control
\(\mathrm{H}^1(Y;\mathcal{P}|_Y[-1])\) when \(\mathcal{P}|_Y[-1]\) is
not perverse.  We rely on \eqref{eq:estimate-ineq} to get a worse
estimate.
\end{remark}

Before giving the proof of the theorem, we collect some preliminary
results that will be used. The full proof appears at the end of this
section (see \ref{situation:proof-1}, \ref{situation:proof-2},
\ref{situation:proof-3}).

\begin{lemma}\label{lemma:generic-euler-characteristic-estimate}
Let \(\mathcal{F} \in D^{b}_{c}(\mathbb{P},\Lambda)\) be a constructible
complex of \(\Lambda\)-sheaves on \(\mathbb{P}\).  Then for the hypersurface
\(Y\) defined by a sufficiently general section of
\(\mathscr{L}^{\otimes d}\) we have
\begin{align*}
  \chi(Y, \mathcal{F}|_{Y}[-1])
  &= r(\mathcal{F})\cdot \delta(\mathcal{F}) \cdot d^{n(\mathcal{F})} + \sum_{j=1}^{n(\mathcal{F})-1} d^{j} \cdot \big( c_{1}(\mathscr{L})^{j} \capprod \mathrm{cc}_{j}(\mathcal{F})\big) \\
  &= r(\mathcal{F}) \cdot \delta(\mathcal{F}) \cdot d^{n(\mathcal{F})} + O(d^{n(\mathcal{F})-1}).
\end{align*}
\end{lemma}

\begin{proof}
By Lemma~\ref{lemma:cc-hypersurface}, we have
\begin{align*}
  \mathrm{cc}_e(\iota_{\ast}\iota^{\ast}\mathcal{F}[-1])
  &= \left[\frac{c_1(\mathscr{L}^{\otimes d})}{1-c_1(\mathscr{L}^{\otimes d})} \capprod \mathrm{cc}(\mathcal{F})\right]_e \\
  &=  \sum_{j = e+1}^{n(\mathcal{F})} d^{j} \cdot c_1(\mathscr{L})^{j-e} \capprod \mathrm{cc}_j(\mathcal{F}).
\end{align*}
When \(e = 0\), applying the index formula, we get
\begin{equation}
\label{eq:characteristic-class-polynomial}
\chi(Y;\mathcal{F}|_Y[-1]) =
d^{n(\mathcal{F})}\cdot  \left( c_{1}(\mathscr{L})^{n(\mathcal{F})} \capprod \mathrm{cc}_{n(\mathcal{F})}(\mathcal{F}) \right) +
\sum_{j = 1}^{n(\mathcal{F}) - 1} d^j \cdot \left( c_1(\mathscr{L})^j\capprod \mathrm{cc}_j(\mathcal{F}) \right).
\end{equation}
By Lemma~\ref{lemma:range-of-cc}, we have
\[
c_1(\mathscr{L})^{n(\mathcal{F})} \capprod \mathrm{cc}_{n(\mathcal{F})} (\mathcal{F}) = r(\mathcal{F}) \cdot \delta(\mathcal{F}).
\]
This completes the proof.
\end{proof}

For later use, we record the following lemma.

\begin{lemma}\label{lemma:higher-intersection-euler}
Let \(A\) be a general codimension \(r\)
linear subspace.
Then for a a general degree \(d\) hypersurface \(Y \subset \mathbb{P}\),
we have
\begin{align*}
  \chi(Y \cap A)
  &= \sum_{i=0}^{n(\mathcal{F})-r} d^i \cdot \sum_{m=r}^{n(\mathcal{F})-i} \binom{m-1}{r-1} \int_{\mathbb{P}} c_1(\mathscr{L})^{m+i} \capprod \mathrm{cc}_{m+i}(\mathcal{F})
\end{align*}
\end{lemma}

\begin{proof}
This is a straightforward computation using Lemma~\ref{lemma:higher-intersection-cc}.
\end{proof}

\begin{lemma}\label{lemma:middle-general-estimate}
Let notation and conventions be as in \ref{setup}. Then there exists a
constant \(C_4\), such that for any sufficiently general degree
\(d\) hypersurface \(Y \subset \mathbb{P}\),
\begin{enumerate}
\item for \(m \geq 0\), we have \(\dim \mathrm{H}^{m}(Y;\mathcal{P}|_Y) = \dim \mathrm{H}^{m+2}(\mathbb{P};\mathcal{P})\);
\item for \(m \leq -2\), we have
\(\dim \mathrm{H}^m(Y;\mathcal{P}|_Y) = \dim \mathrm{H}^m(\mathbb{P};\mathcal{P})\); and
\item in middle dimension, we have
\(|\dim \mathrm{H}^{-1}(Y; \mathcal{P}|_{Y}) - r(\mathcal{P}) \cdot \delta(\mathcal{P}) \cdot d^{n(\mathcal{P})}| \leq C_4 d^{n(\mathcal{P})-1}\),
\end{enumerate}
\end{lemma}

\begin{proof}
Since \(Y\) is sufficiently general, and \(\mathcal{P}\) is perverse
(hence in particular perverse \emph{connective}), the perverse weak
Lefschtz theorem, Lemma~\ref{lemma:gysin}, implies the Gysin map
\[
\mathrm{H}^{m-2}(Y;\mathcal{P}|_{Y}(-1)) \to \mathrm{H}^{m}(\mathbb{P};\mathcal{P}),
\]
is surjective for \(m = 1\), and bijective for \(m \geq 2\).  This
proves the first assertion.  The second assertion follows from the
coconnectiveness of \(\mathcal{P}\) and the usual weak Lefscehtz
theorem, Lemma~\ref{lemma:weak-lefschetz}.

The last assertion follows from
Lemma~\ref{lemma:generic-euler-characteristic-estimate} and the previous two
assertions.  Indeed, for \(Y\) general (i.e.,
\(\mathrm{SS}(\mathcal{P})\)-transverse) we have
\begin{align*}
  \chi(Y; \mathcal{P}|_{Y}[-1])
  &= \dim \mathrm{H}^{-1}(Y;\mathcal{P}|_{Y})  - \sum_{m\leq -2} (-1)^{m} \dim \mathrm{H}^{m}(\mathbb{P};\mathcal{P}) - \sum_{m\geq 0} (-1)^{m} \dim \mathrm{H}^{2+m}(\mathbb{P};\mathcal{P}) \\
  &= \dim \mathrm{H}^{-1}(Y;\mathcal{P}|_{Y}) - \chi(\mathbb{P}; \mathcal{P})  - \dim\mathrm{H}^{-1}(\mathbb{P}; \mathcal{P}) +
  \dim\mathrm{H}^{0}(\mathbb{P}; \mathcal{P}) - \dim\mathrm{H}^{1}(\mathbb{P}; \mathcal{P}).
\end{align*}
We win by applying Lemma~\ref{lemma:generic-euler-characteristic-estimate}.
In fact, we can take
\begin{equation}\label{eq:C4}
C_4 = \max\left\{ {\textstyle \int_{\mathbb{P}} c_1(\mathscr{L})^j \capprod \mathrm{cc}_j(\mathcal{P}) : j \geq 0}\right\} + \dim \mathrm{H}^{-1}(\mathbb{P};\mathcal{P}) + \dim \mathrm{H}^1(\mathbb{P};\mathcal{P})
- \dim \mathrm{H}^0(\mathbb{P};\mathcal{P}). \qedhere
\end{equation}
\end{proof}

\begin{situation}[Proof of Theorem~\ref{theorem:main}(1)]
\label{situation:proof-1}
Suppose \(Y = \{f = 0\}\) where
\(f \in \mathrm{H}^0(\mathbb{P};\mathscr{L}^{\otimes d})\).  Let \(g\)
be a general section of
\(\mathrm{H}^0(\mathbb{P};\mathscr{L}^{\otimes d})\) defining a
hypersurface \(Y_{\infty}\) such that \(Y_{\infty}\) is general in the
sense that Lemma~\ref{lemma:middle-general-estimate} holds true.
Consider the pencil of hypersurfaces \(Y_t\) defined by
\(f + tg = 0\), where \(t \in \mathbb{P}^1\).  Then for \(t\) general,
Lemma~\ref{lemma:middle-general-estimate} remains true.  By the
Proper Degeneration Lemma~\eqref{eq:proposition},
applying to the perverse sheaf \(\mathcal{P}\) (which is in particular
perverse coconnective) we conclude that
\begin{align}
  \dim \mathrm{H}^{-1}(Y;\mathcal{P}|_Y)
  &\leq  \dim \mathrm{H}^{-1}(Y_t;\mathcal{P}|_{Y_t}) +
    \dim \mathrm{H}^{-2}(Y_t;\mathcal{P}|_{Y_{t}}) + \dim \mathrm{H}^{-1}(\mathbb{P};\mathcal{P}) \nonumber
  \\
  &=  \dim \mathrm{H}^{-1}(Y_t;\mathcal{P}|_{Y_t}) +
    \dim \mathrm{H}^{-2}(\mathbb{P};\mathcal{P}) + \dim \mathrm{H}^{-1}(\mathbb{P};\mathcal{P}) \nonumber
  \\
  &= \chi(Y_t; \mathcal{P}|_{Y_t}[-1]) + \chi(\mathbb{P};\mathcal{P}) + 2\dim \mathrm{H}^{-1}(\mathbb{P};\mathcal{P}) + \dim \mathrm{H}^1(\mathbb{P};\mathcal{P}) \label{eq:explicit-1}\\
  &\quad - \dim \mathrm{H}^0(\mathbb{P};\mathcal{P}) + \dim \mathrm{H}^{-2}(\mathbb{P};\mathcal{P})\nonumber \\
  &\leq r(\mathcal{P}) \cdot \delta(\mathcal{P}) \cdot d^{n(\mathcal{P})} + C_1 d^{n(\mathcal{P})-1}
    \nonumber
\end{align}
where the second step uses
Lemma~\ref{lemma:middle-general-estimate}(2), and we can take the
constant \(C_1\) to be
\begin{equation}\label{eq:C1}
C_1 = C_4 + \dim \mathrm{H}^{-2}(\mathbb{P};\mathcal{P}) + \dim \mathrm{H}^{-1}(\mathbb{P};\mathcal{P})
\end{equation}
where \(C_4\) is from Lemma~\ref{lemma:middle-general-estimate}(3).
\qed
\end{situation}

\begin{situation}[Proof of Theorem~\ref{theorem:main}(2)]
\label{situation:proof-2}
Let us choose an auxiliary flag
\[
A_{\bullet}: \quad
\mathbb{P} = A_0 \supset A_1 \supset \cdots \supset A_i \supset \cdots
\]
of general linear subspaces in \(\mathbb{P}\) with
\(\dim A_i = \dim \mathbb{P} - i\).  Here, by a linear subspace, we
mean an intersection of divisors in the linear system
\(\mathrm{H}^0(\mathbb{P}; \mathscr{L})\).  Since the flag
\(A_{\bullet}\) is general
\(\mathcal{P}_i = \mathcal{P}|_{A_{i}}[-i]\) is a perverse sheaf on
\(A_i\) (Lemma~\ref{lemma:general-hyperplane-preserve-perv}), and
\(n(\mathcal{P}_i) = n(\mathcal{P}) - i\).  By
Theorem~\ref{theorem:main}(1), applied to \(A_{i}\) and
\(\mathcal{P}_i\), we can find a constant \(C^{(A_{\bullet})}\) independent of
\(Y\), such that for any hypersurface \(Y \subset \mathbb{P}\) of
degree \(d\), we have
\[
\dim \mathrm{H}^{-1}(Y \cap A_{i}; \mathcal{P}_i|_{Y\cap A_{i}}) \leq r(\mathcal{P}_i)
\cdot \delta(\mathcal{P}_i) \cdot d^{n(\mathcal{P}) - i} + C^{(A_{\bullet})}
d^{n(\mathcal{P}) - i - 1}.
\]

Since the hypothesis is \(\mathcal{P}|_Y[-1]\) is a perverse sheaf, it
is in particular perverse connective.  Since the flag \(A_{\bullet}\)
is chosen in a generic fashion, the perverse weak Lefschetz theorem
(Lemma~\ref{lemma:gysin}) implies that for any \(e\geq 1\), there is a
surjective Gysin map
\[
\begin{tikzcd}
\mathrm{H}^{-1}(Y\cap A_e; \mathcal{P}_e|_{Y\cap A_e})= \mathrm{H}^{-e}(Y\cap A_e; \mathcal{P}|_{Y\cap A_{e}}[-1]) \ar[r,two heads] & \mathrm{H}^{e}(Y; \mathcal{P}|_Y[-1]).
\end{tikzcd}
\]
Therefore, for each \(e\geq 1\), we have
\begin{align*}
  \dim \mathrm{H}^{-1+e}(Y;\mathcal{P}|_Y)
  &\leq r(\mathcal{P}_e) \cdot \delta(\mathcal{P}_e) \cdot d^{n(\mathcal{P}) - e} + C_1^{(A_{\bullet})} \cdot d^{n(\mathcal{P})-e-1} \\
  & \leq C^{(e)} d^{n(\mathcal{P}) - e}.
\end{align*}
Here, the constant \(C^{(e)}\) depends only on the rank and degree of
\(\mathcal{P}\) restricted to a generic codimension \(e\) linear
subspace, is independent of the choice of \(Y\).

By the traditional weak Lefschetz theorem,
Lemma~\ref{lemma:weak-lefschetz}, we have
\(\mathrm{H}^m(Y;\mathcal{P}|_Y)\cong\mathrm{H}^m(\mathbb{P};\mathcal{P})\)
if \(m \leq -2\).  Hence they will only contribute an \(O(1)\) in the
Betti sum.  Combining Theorem~\ref{theorem:main}(1) and the preceding
discussions, Theorem~\ref{theorem:main}(2) is proved.\qed
\end{situation}

\begin{situation}[Proof of Theorem~\ref{theorem:main}(3)]
\label{situation:proof-3}
Let \(Y = Y_0 \subset X\) be an arbitrary degree \(d\) hypersurface.
Then by Theorem~\ref{theorem:main}(1), we already have
\[
\dim \mathrm{H}^{-1}(Y;\mathcal{P}|_Y) \leq
r(\mathcal{P}) \cdot \delta(\mathcal{P}) \cdot d^{n(\mathcal{P})} + C_1 \cdot d^{n(\mathcal{P})-1}.
\]
As in the previous discussion, for \(e \leq -2\),
\(\dim \mathrm{H}^{-e}(Y;\mathcal{P}|_Y) = \dim \mathrm{H}^{-e}(\mathbb{P},\mathcal{P})\),
so they only contribute an \(O(1)\) in the final estimate.

So we are left to estimate \(\mathrm{H}^e(Y;\mathcal{P}|_{Y})\) for
\(e\geq 0\).  Since \(\mathcal{P}|_Y[-1]\) may not be perverse
connective, the argument in \ref{situation:proof-2} breaks down as we
cannot apply the perverse weak Lefschetz theorem to
\(\mathcal{P}|_Y[-1]\).  Therefore we need a to apply the weaker, but
unconditional degeneration inequality \eqref{eq:estimate-ineq}.
Let us first deal with the case where \(e = 0\).
\end{situation}

\begin{lemma}\label{lemma:bad-factor}
Notation and conventions be as in \ref{setup}.  Suppose
\(\mathcal{P}\) is a perverse sheaf on \(\mathbb{P}\).  Then for any
degree \(d\) hypersurface \(Y\subset \mathbb{P}\), there is a constant
\(C_0\) independent of \(Y\), such that
\begin{align*}
  \mathrm{H}^0(Y;\mathcal{P}|_Y)
  &\leq 2 \cdot r(\mathcal{P}) \cdot \delta(\mathcal{P}) \cdot d^{n(\mathcal{P})}
    + C_0 \cdot d^{n(\mathcal{P}) -1}.
\end{align*}
\end{lemma}

\begin{proof}
Let \(g \in \mathrm{H}^0(\mathbb{P};\mathscr{L})\) be a general linear
form.  Let \(L = \{g = 0\}\), so that \(L\) is
\(\operatorname{SS}(\mathcal{P})\)-transverse.  Let
\(f \in \mathrm{H}^0(\mathbb{P};\mathscr{L}^{\otimes d})\) be the
defining section of \(Y\).  Without loss of generality, we may assume
that \(Y\) is reduced; otherwise, we reduce to a lower degree case.

Consider the pencil of hypersurfaces \(Y_t\) defined by
\(\{f + t g^d = 0\}\).  Since \(Y\) is reduced, \(Y_t\) is also
reduced for general \(t\).  Because \(Y_{\infty} = \{g = 0\}\) is
\(\operatorname{SS}(\mathcal{P})\)-transverse, \(Y_t\) is
\(\mathrm{SS}(\mathcal{P})\)-transverse for general \(t\).
In particular,
\(\mathcal{P}|_{Y_t}[-1]\) is perverse for general \(t\).
(This can also be shown directly, see \cite[Lemma~III~6.3]{kiehl-weissauer_weil-conjectures-perverse-sheaves}.)
By the Projective Degeneration Lemma~\eqref{eq:estimate-ineq}, we
have, for general \(t\),
\begin{align*}
  \dim \mathrm{H}^1(Y; \mathcal{P}[-1])
  &\leq \dim \mathrm{H}^1(Y_t; \mathcal{P}[-1])
    + 2 \dim \mathrm{H}^0(Y_t; \mathcal{P}[-1])
    + \dim \mathrm{H}^0(\mathbb{P}; \mathcal{P})\\
    &\quad + \dim \mathrm{H}^0(B; \mathcal{P}[-2]|_B) \\
  &\leq 2 \cdot r(\mathcal{P}) \cdot \delta(\mathcal{P}) \cdot d^{n(\mathcal{P})}
    + 2C_2 \cdot d^{n(\mathcal{P})-1}\\
  &\quad + \dim \mathrm{H}^{0}(\mathbb{P}; \mathcal{P})
    + \dim \mathrm{H}^0(B; \mathcal{P}[-2]|_B).
\end{align*}
Here \(B = Y_0 \cap Y_{\infty} = Y \cap L\), and the second inequality
follows from Theorem~\ref{theorem:main}(2).  Since \(B\) is a degree \(d\)
hypersurface in \(L = \{g = 0\}\), we may use Theorem~\ref{theorem:main}(1)
for \(L\) and \(\mathcal{P}|_{L}[-1]\).  Note that \(L\) is independent of
\(d\) and is chosen generically so that \(\mathcal{P}|_L[-1]\) has support
of one lower dimension.  Therefore, this part will contribute an
\(O(d^{n-1})\) term in the formula, as required.
\end{proof}

\begin{proof}[End of Proof of Theorem~\ref{theorem:main}(3)]
We now use a similar strategy as in \ref{situation:proof-2} to treat
\(\mathrm{H}^e(Y;\mathcal{P}|_Y)\) for \(e \geq 1\).  By
Lemma~\ref{lemma:estimate-perv-dergee}, although
\(\mathcal{P}|_{Y}[-1]\) may no longer be connective,
\(\mathcal{P}|_Y\) remains
connective.  Therefore, the perverse weak Lefschetz theorem still
applies to \(\mathcal{P}|_Y\) (although not to \(\mathcal{P}|_Y[-1]\)).  As in
\ref{situation:proof-2}, we choose a general flag of linear subspaces
\[
\mathbb{P} = A_0 \supset A_1 \supset \cdots \supset A_i \supset \cdots
\]
with \(\dim A_i = \dim \mathbb{P} - i\).  By Lemma~\ref{lemma:gysin},
the perverse weak Lefschetz theorem gives us surjective Gysin maps
\[
\mathrm{H}^1(Y \cap A_e; \mathcal{P}_e|_{Y \cap A_e}[-1]) =
\mathrm{H}^{-e}(Y \cap A_e; \mathcal{P}|_{Y \cap A_e})
\to \mathrm{H}^e(Y; \mathcal{P}|_Y),
\]
where \(\mathcal{P}_e = \mathcal{P}|_{A_e}[-e]\) is a perverse sheaf
on \(A_e\).  For \(e \geq 1\), we may apply
Lemma~\ref{lemma:bad-factor} to \(A_e\) and \(\mathcal{P}_e\), and
find a constant \(C^{(A)}_0\) such that
\begin{align*}
  \dim \mathrm{H}^{1+e}(Y;\mathcal{P}[-1])
  &\leq \dim \mathrm{H}^{1}(Y\cap A_e;
    \mathcal{P}_e|_{Y\cap A_e}[-1]) \\
  &\leq 2 r(\mathcal{P}) \cdot \delta(\mathcal{P})\cdot d^{n(\mathcal{P})-e}
    + C^{(A)}_0 \cdot d^{n(\mathcal{P})-e-1}.
\end{align*}
Therefore, for \(e \geq 1\),
\(\sum_{e \geq 1} \dim \mathrm{H}^{1+e}(Y; \mathcal{P}|_Y[-1])\) is
bounded by \(O(d^{n(\mathcal{P})-1})\).  In summary,
\begin{align*}
B(Y, \mathcal{P}|_Y)
&= \underbrace{O(1)}_{\mathrm{H}^{\leq-1}(Y;\mathcal{P}|_Y[-1])}
+ \dim \mathrm{H}^0(Y; \mathcal{P}[-1])
+ \dim \mathrm{H}^1(Y; \mathcal{P}[-1]) +
\underbrace{O(d^{n(\mathcal{P})-1})}_{\mathrm{H}^{\geq2}(Y;\mathcal{P}|_Y[-1])} \\
&= 3 r(\mathcal{P}) \delta(\mathcal{P}) d^{n(\mathcal{P})}
+ O(d^{n(\mathcal{P})-1}).
\end{align*}
This completes the proof of the theorem.
\end{proof}

For possible future use, and to prepare for proving the uniformity of
these bounds as \(\ell\) varies, we record here some details regarding
the bounds obtained in Theorem~\ref{theorem:main}.

\begin{situation}[Explication of Theorem~\ref{theorem:main}(1)]
\label{situation:explication-1}
Let \(Y\) be an arbitrary degree \(d\) hypersurface in \(\mathbb{P}\).
By \eqref{eq:C4} and \eqref{eq:explicit-1}, we have
\[
\dim \mathrm{H}^{0}(Y;\mathcal{P}|_{Y}[-1])
\leq \sum_{j=0}^{n(\mathcal{P})} d^{j} \cdot \int_{\mathbb{P}}
c_{1}(\mathscr{L})^{j} \capprod \mathrm{cc}_{j}(\mathcal{P})
+ 2 B(\mathbb{P},\mathcal{P}).
\]
\end{situation}

\begin{situation}[Explication of Theorem~\ref{theorem:main}(2)]
\label{situation:explication-2}
For convenience, we write \(n = n(\mathcal{P})\).  In the space of all
complete flags of linear subspaces in \(\mathbb{P}\), we choose a
general flag \(A_{\bullet}\).  If \(\mathcal{P}|_Y[-1]\) is a perverse
sheaf,  so is \(\mathcal{P}_r|_{Y \cap A_r}[-1]\),
where \(\mathcal{P}_r \coloneq \mathcal{P}|_{A_r}[-r]\).
Then for any \(r \geq 1\), the argument in
\ref{situation:proof-2}, together with
Lemma~\ref{lemma:higher-intersection-euler} and
\ref{situation:explication-1}, shows that
\begin{align*}
  \label{eq:explicit-bound-perverse}
  \dim \mathrm{H}^r(Y; \mathcal{P}|_Y[-1])
  & \leq \dim \mathrm{H}^0(Y \cap A_r; \mathcal{P}_r|_{Y \cap A_r}[-1]) \\
  & \leq \sum_{i=0}^{n-r} d^i
    \sum_{m=1}^{n-i} \binom{m-1}{r-1}
    \int_{\mathbb{P}}
      c_1(\mathscr{L})^{m+i} \capprod \mathrm{cc}_{i+m}(\mathcal{P}) \\
  &\quad + 2 B(A_{r}, \mathcal{P}|_{A_{r}}).
\end{align*}
\end{situation}

\begin{situation}[Explication of Theorem~\ref{theorem:main}(3)]
\label{situation:explication-3}
We continue to select a general flag \(A_{\bullet}\) as in
\ref{situation:explication-2}.  Now assume that
\(\mathcal{P}|_Y[-1]\) is not a perverse sheaf.  We retain the bound
for \(\mathrm{H}^{0}(Y;\mathcal{P}|_{Y}[-1])\) given in
\ref{situation:explication-1}.

Recall from Lemma~\ref{lemma:bad-factor} that \(Y_{\infty}\) is a
general hyperplane.  Accordingly, \(B\) is a degree \(d\) hypersurface
in a general hyperplane.  Applying the result of
\ref{situation:explication-2}, we obtain
\[
\mathrm{H}^0(B; \mathcal{P}|_{B}[-2]) \leq
\sum_{j=0}^{n-1} d^j \sum_{m=1}^{n-j}
\int_{\mathbb{P}} c_1(\mathscr{L})^{m+j} \capprod
\mathrm{cc}_{j+m}(\mathcal{P}) + 2 B(A_{1}, \mathcal{P}|_{A_{1}}).
\]
This is the same as the bound for
\(\mathrm{H}^1(Y_t; \mathcal{P}_t|_{Y_t}[-1])\) for general \(t\),
where \(\mathcal{P}|_{Y_t}[-1]\) is perverse, as given in
\ref{situation:explication-2}.  Therefore,
\begin{align*}
  \dim \mathrm{H}^{1}(Y; \mathcal{P}|_{Y}[-1])
  &\leq \dim \mathrm{H}^{1}(Y_t) + 2 \dim \mathrm{H}^0(Y_t; \mathcal{P}_t|_{Y_t}[-1]) \\
  &\quad + \dim \mathrm{H}^0(B; \mathcal{P}|_{B}[-2]) + \dim \mathrm{H}^0(\mathbb{P}; \mathcal{P}) \\
  &\leq 2 \sum_{i=0}^{n-1} d^i \sum_{m=1}^{n-i}
    \int_{\mathbb{P}} c_1(\mathscr{L})^{m+i} \capprod \mathrm{cc}_{i+m}(\mathcal{P})
   + 4 B(A_{1}, \mathcal{P}|_{A_{1}}) \\
  &\quad + 2 \sum_{j=0}^{n} d^{j}
    \int_{\mathbb{P}} c_1(\mathscr{L})^{j} \capprod \mathrm{cc}_{j}(\mathcal{P})
   + 3 B(\mathbb{P},\mathcal{P}).
\end{align*}

For \(r \geq 2\), we can invoke the weak Lefschetz theorem and the
preceding results, obtaining
\begin{align*}
  \dim \mathrm{H}^{r}(Y; \mathcal{P}|_Y[-1])
  &\leq \dim \mathrm{H}^1(Y \cap A_r;
    \mathcal{P}_{r-1}|_{Y \cap A_{r-1}}[-1]) \\
  &\leq 2 \sum_{j=0}^{n-r-1} d^j \sum_{m=1}^{n-i} \binom{m-1}{r-1}
    \int_{\mathbb{P}} c_1(\mathscr{L})^{m+j} \capprod
    \mathrm{cc}_{j+m}(\mathcal{P}) \\
  &\quad + 4 B(A_{r+1}, \mathcal{P}|_{A_{r+1}}) \\
  &\quad + 2 \sum_{j=0}^{n-r} d^j \sum_{m=1}^{n-i} \binom{m-1}{r-2}
    \int_{\mathbb{P}} c_1(\mathscr{L})^{m+j} \capprod
    \mathrm{cc}_{j+m}(\mathcal{P}) \\
  &\quad + 3 B(A_r, \mathcal{P}|_{A_r}).
\end{align*}
\end{situation}

\begin{corollary}\label{corollary:bound-any}
Let \(\mathcal{F}\) be an arbitrary constructible complex on
\(\mathbb{P}\).  Set
\begin{itemize}
\item \(r = \sum_{e\in \mathbb{Z}} r({}^{\mathrm{p}}\mathcal{H}^{e}\mathcal{F})\),
\item
\(\delta = \sum_{e\in \mathbb{Z}}\deg\operatorname{Supp}({}^{\mathrm{p}}\mathcal{H}^{e}(\mathcal{F}))\),
and
\item \(n = \dim \operatorname{Supp}(\mathcal{F})\).
\end{itemize}
Then there exists \(C > 0\) such that for any degree \(d\)
hypersurface in \(\mathbb{P}\), we have
\[
B(Y,\mathcal{F}) \leq 3r\delta d^{n} + C d^{n-1}
\]
\end{corollary}

\begin{proof}
Apply Theorem~\ref{theorem:main}(3) and Lemma~\ref{lemma:trivial} to
\({}^{\mathrm{p}}\mathcal{H}^{e}(\mathcal{F})\) for each \(e\), noting
that \(n({}^{\mathrm{p}}\mathcal{H}^{e}(\mathcal{F})) \leq n\).
\end{proof}

\section{Betti number estimates}
\label{sec:applications}

In this section, we use Theorem~\ref{theorem:main} to deduce the
asymptotic of total Betti numbers of a degree \(d\) hypersurface in
any projective variety.

\begin{situation}[Setup]
Throughout this section, let \(X\) be a possibly singular projective
variety.  Let \(\mathscr{L}\) be a very ample invertible sheaf on \(X\).
We shall simply refer to the zero locus of a nonzero section of
\(\mathrm{H}^0(X;\mathscr{L}^{\otimes d})\) as a degree \(d\)
hypersurface.
\end{situation}

Our goal is to prove two theorems regarding the asymptotic properties
of the total Betti numbers of hypersurfaces in \(X\).  The first one
is surprisingly general, regardless how singular the ambient variety
\(X\) is.

\begin{theorem}\label{theorem:arbitrary}
Suppose \(\dim X = n\), and suppose \(\mathcal{F}\) is
a plain constructible sheaf on \(X\).  Then there exist a constant \(C> 0\) such
that for any degree \(d\) hypersurface \(Y \subset X\) (\(C\) is
independent of \(Y\)), we have
\[
B(Y;\mathcal{F}|_{Y}) \leq 3 \operatorname{rank}^{\circ}(\mathcal{F}) \cdot \deg(X) \cdot d^n + C \cdot d^{n-1}
\]
(for the definition of the generic rank
\(\operatorname{rank}^{\circ}\), see
Definition~\ref{definition:vir-gen-rank}).
\end{theorem}

\begin{proof}
We proceed by an induction on the dimension of \(X\).  When the
dimension is zero, there is nothing to prove.  Suppose for any polarized variety
\((D,\mathscr{L})\) of dimension \(\leq n-1\), and any constructible sheaf \(\mathcal{F}\)
on \(D\), we have found a constant \(C_{\mathscr{L}}(D,\mathcal{F})\) which makes
the inequality holds.

Let
\[
\iota\colon X \to \mathbb{P} = \operatorname{Proj}\mathrm{Sym}^{\bullet}\mathrm{H}^0(X;\mathscr{L})^{\vee}
\]
be the projective embedding induced by \(\mathscr{L}\).  Let
\(j\colon X^{\circ} \to X\) be an open immersion of satisfying the
following properties:
\begin{itemize}
\item \(X^{\circ}\) is smooth,
\item \(X^{\circ}\) contains are the generic points of the
\(n\)-dimensional irreducible components of \(X\),
\item \(\mathcal{F}|_{X^{\circ}}\) is lisse of rank equal to
\(\mathrm{rank}^{\circ}(\mathcal{F})\).
\end{itemize}
By shrinking \(X^{\circ}\), we may further assume that
\begin{itemize}
\item \(X \setminus X^{\circ}\) is a Cartier divisor.
\end{itemize}
Let \(i \colon D \to X\) be the inclusion of
the complement \(D = X \setminus X^{\circ}\) of \(X^{\circ}\).  Then
we have a distinguished triangle
\[
j_! j^{\ast}\mathcal{F}[n] \to \mathcal{F}[n] \to i_{\ast} i^{\ast} \mathcal{F}[n].
\]
Since \(D\) is a Cartier divisor, \(j\) is quasi-finite and affine.
Therefore, \(j_!j^{\ast}\mathcal{F}[n]\), hence \(\iota_{\ast} (j_!j^{\ast}\mathcal{F}[n])\), is perverse.

Let \(Y\) be any degree \(d\) hypersurface in \(X\).
Let \(H \subset \mathbb{P}\) be the degree \(d\) hypersurface in \(\mathbb{P}\) such that \(Y = X \cap H\).
Then by
Theorem~\ref{theorem:main}(3) that we have
\[
B(Y, j_!j^{\ast}\mathcal{F}[n]|_Y) = B(H, \iota_{\ast} (j_!j^{\ast}\mathcal{F}[n])|_{H})
\leq 3\operatorname{rank}^{\circ}(\mathcal{F}) \cdot \deg(\overline{X}{}^{\circ})
\cdot d^n + C_3 \cdot d^{n-1}
\]
for some constant \(C_3\).  Here, \(\overline{X}{}^{\circ}\) is the
closure of \(X^{\circ}\).  The result then follows from an induction
on dimension: because
\[
B(Y, \mathcal{F}|_{Y}) \leq B(Y, j_!j^{\ast}\mathcal{F}[n]|_Y) + B(Y \cap D, i^{\ast} \mathcal{F}|_{Y\cap D}),
\]
we can simply take
\[
C = C_{\mathscr{L}}(X,\mathcal{F})=C_{3} + \max\left\{3\operatorname{rank}^{\circ}(i^{\ast}\mathcal{F})\cdot \deg(D), C_{i^{\ast}\mathscr{L}}(D,i^{\ast}\mathcal{F})\right\}.
\]
This completes the proof.
\end{proof}

The unpleasant factor \(3\) can be saved if \(X\) is a set-theoretically
a local complete intersection.

\begin{definition}
Let \(V\) be a variety over \(k\) of pure dimension \(n\).  We say
\(V\) is \emph{set-theoretically a local complete intersection} if
there exists an open affine cover \(V = \bigcup_{m\in I} V_{m}\) and
with the following property: for each \(m \in I\), there exists an integer \(r\geq 0\), and polynomials \(f_{1},\ldots, f_{r} \in k[T_{1},\ldots,T_{n+r}]\) such that
\[
V_{m,\mathrm{red}} \cong \operatorname{Spec}\frac{k[T_1,\ldots,T_{n+r}]}{\sqrt{(f_{1},\ldots,f_{r})}}.
\]
Clearly, any local complete intersection variety is set-theoretically a local complete intersection.
\end{definition}

\begin{lemma}\label{lemma:lci-perv}
Let \(V\) be a variety over \(k\) of pure dimension \(n\)
and let \(\mathcal{E}\) be a lisse sheaf of \(\Lambda\)-modules.
If \(V\) is set-theoretically a local complete intersection,
then \(\Lambda_V[n]\) is a perverse sheaf.
\end{lemma}

\begin{proof}
Since \(V\) and \(V_{\mathrm{red}}\) have equivalent étale topoi,
and since the assertion is local on \(V\), we may suppose \(V\) is realized
as a complete intersection in \(\mathbb{A}^{n+r}\) defined by \(r\)
polynomials \(f_1, \ldots, f_r \in k[T_1, \ldots, T_{n+r}]\).  By
successively applying Lemma~\ref{lemma:estimate-perv-dergee} for each
defining equation, we conclude that \(\Lambda_V[n]\) is perverse
coconnective.  Moreover, as it is simply a shift of the constant sheaf
by \(n\), perverse connectivity holds automatically. The result
follows.
\end{proof}

\begin{lemma}\label{lemma:const-vs-local-const-perv}
Let \(V\) be a variety over \(k\) of pure dimension \(n\), and suppose
that \(\Lambda_V[n]\) is perverse.  Then, for any lisse sheaf
\(\mathcal{E}\) of \(\Lambda\)-modules on \(V\), the shifted sheaf
\(\mathcal{E}[n]\) is also perverse.
\end{lemma}

\begin{proof}
If \(\Lambda\) is infinite, we may select an integral lattice in
\(\mathcal{E}\) and reduce to the finite case, so we may assume
without loss of generality that \(\Lambda\) is finite.  Furthermore,
since the property in question depends only on the reduced subscheme
underlying \(V\), we may assume that \(V\) is a local complete
intersection.

Given that \(\mathcal{E}\) is a lisse sheaf of \(\Lambda\)-modules
and \(\Lambda\) is finite, there exists a finite étale cover
\(\pi\colon V^{\prime} \to V\) such that \(\pi^*\mathcal{E}\) is a
constant sheaf, i.e.,
\(\pi^*\mathcal{E} \simeq \Lambda_{V^{\prime}}^{\oplus m}\).  The
morphism \(\pi\) being étale implies that \(V^{\prime}\) is itself a
local complete intersection and that the pullback functor \(\pi^*\) is
t-exact for the perverse t-structure.  Thus, \(\pi^*\mathcal{E}[n]\),
being a direct sum of \(\Lambda_{V^{\prime}}[n]\), is perverse
coconnective.  It follows from Lemma~\ref{lemma:check-connectivity}
that \(\mathcal{E}[n]\) is perverse coconnective on \(V\).  Finally,
as \(\dim V = n\) and \(\Lambda_V[n]\) is perverse connective by
assumption, the conclusion follows.
\end{proof}

\begin{theorem}\label{theorem:lci}
Let \(X\) be an irreducible projective variety of pure dimension
\(n\).  Let \(\mathscr{L}\) be a very ample line invertible sheaf on
\(X\).  Suppose \(\Lambda_X[n]\) is a perverse sheaf on \(X\).  (This
is the case, e.g., when \(X\) is set-theoretically a local complete
intersection, by Lemma~\ref{lemma:lci-perv}).

Then for any lisse sheaf \(\mathcal{E}\) on \(X\), there is a constant \(C\) (independent of the coefficient field
\(\Lambda\)), such that for any degree \(d\) hypersurface
\(Y \subset X\), we have
\[
B(Y, \mathcal{E}|_{Y}) \leq \deg(X) \cdot \mathrm{rank}(\mathcal{E}) \cdot d^n + C\cdot d^{n-1}.
\]
\end{theorem}




\begin{proof}
By Lemma~\ref{lemma:lci-perv}, we see \(\mathcal{E}[n]\) is
a perverse sheaf.

Let \(\iota\colon X \to \mathbb{P}\) be the projective embedding
defined by \(\mathscr{L}\).  Let \(Y\) be a hypersurface of degree
\(d\).  Since \(X\) is \(n\)-dimensional and irreducible,
\(\dim Y = n-1\), and, by the Bertini theorem,
unless \(n=1\), \(Y\) remains irreducible.
This implies that \(\mathcal{E}|_{Y}[n-1]\) is necessarily a perverse
sheaf (Lemma~\ref{lemma:general-hyperplane-preserve-perv}).
We win by applying Theorem~\ref{theorem:main}(2) to the
perverse sheaf \(\iota_{\ast}\mathcal{E}[n]\).
\end{proof}

\section{Uniformity in \texorpdfstring{\(\ell\)}{ℓ}}
\label{sec:l-independence}

In this section, we show that if \(V\) can be defined over a finite
field, and the perverse sheaf \(\mathcal{P}\) is part of a compatible
system of perverse sheaves, then the constants in the upper bounds of
Theorem~\ref{theorem:main} can be made independent of \(\ell\).

\begin{situation}[Conventions]
In this section, we work over a finite base field \(\mathbf{F}_q\) of characteristic \(p\) and
let \(k\) be an algebraic closure of \(\mathbb{F}_{q}\).  We use
the following notation system:
capital Latin letters with a subscript \(0\), such as
\(X_0, S_0\), denote separated schemes of finite type defined over
\(\mathbf{F}_q\).  Removing the subscript (e.g., \(X, S\)) indicates
the base change of these schemes to a fixed algebraic closure
\(\overline{\mathbf{F}}_q\).  If
\(\mathcal{F} \in D^{b}_{c}(X_{0},\Lambda)\), we shall still denote by
\(\mathcal{F}\) its inverse image by the morphism \(X \to X_{0}\).
This will hardly cause any confusions.
\end{situation}

\begin{definition}[cf.~\cite{fujiwara_gabber-independence-of-l-for-intersection-cohomology}]
Let \(E\) be a field of characteristic \(0\).
Let \(I\) be a subset of
\[
\{(\ell,\iota) : \ell \neq p \text{ is a prime}, \iota\colon E \hookrightarrow \overline{\mathbb{Q}}_{\ell} \text{ is a field embedding}\}.
\]
For \(\alpha \in I\), denote by \(\ell_{\alpha}\) the first coordinate
of \(\alpha\).

Let \(X_{0}\) be a separated scheme of finite type over
\(\mathbb{F}_{q}\).  We say a system
\[
(\mathcal{F}_{\alpha})_{\alpha\in I} \in \prod_{\alpha\in I}
D^{b}_{c}(X_{0},\overline{\mathbb{Q}}_{\ell_{\alpha}})
\]
of constructible complexes on a finite type \(\mathbb{F}_{q}\)-scheme
\(X_{0}\) is \emph{weakly \((E,I)\)-compatible}, if
\begin{itemize}
\item the local factor
\[
\det(1 - t\mathrm{Frob}_{x},\mathcal{F}_{\alpha}) =
\prod_{i\in \mathbb{Z}} \det(1 - t\mathrm{Frob}_{x},\mathcal{H}^{i}(\mathcal{F}_{\alpha}))^{(-1)^{i-1}}
\]
lies in \(1 + tE[t]\)
for any closed point \(x\) of \(X\), where \(\mathrm{Frob}_{x}\) is the geometric Frobenius automorphism of \(\mathrm{Gal}(k/\kappa(x))\).
\item for any closed point \(x\) of \(X\), the polynomials
\(\det(1 - t\mathrm{Frob}_{x},\mathcal{F}_{\alpha})\) are all equal
for all \(\alpha \in I\).
\end{itemize}
\end{definition}

\begin{lemma}\label{lemma:euler-characteristic-independent-of-l}
Let \(X_0\) be a separated scheme of finite type over
\(\mathbb{F}_{q}\).  Let \((\mathcal{F}_{\alpha})_{i\in I}\) be a
weakly \((E,I)\)-compatible system.  Then
\(\chi(X,\mathcal{F}_{\alpha})\) is independent of \(\alpha\in I\).
\end{lemma}

\begin{proof}
This follows immediately from the Grothendieck trace formula.  Indeed,
by the compatibility, the L-functions
\[
L(\mathcal{F}_{\alpha},t) = \prod_{x\in|X|}
\det(1-t\mathrm{Frob}_{x},\mathcal{F}_{\alpha})^{-1}
\]
lies in \(1+tE\llbracket{t}\rrbracket\) independent of \(\ell\).  By
the trace formula, \(L(\mathcal{F},t)\) is a rational function in
\(E(t)\), and \(\chi(X,\mathcal{F}_{\alpha})\) equals the degree
of the denominator of \(L(\mathcal{F}_{\alpha},t)\) minus the degree
of the numerator of \(L(\mathcal{F}_{\alpha},t)\), hence is also
independent of \(\ell\).
\end{proof}

\begin{situation}
[Construction]
\label{situation:flag-construction}
Let \(\mathcal{M}\) denote the space of full flags in the projective space
\(\mathbb{P}^{N}\).  Let \(\mathcal{A}^{(e)} \to \mathcal{M}\) be the
universal codimension \(e\) linear subspace.  Let \(\overline{\eta}\) be a
geometric generic point of \(\mathcal{M}\).  For any locally closed subset
\(X \subset \mathbb{P}^{n}\), we consider the following commutative diagram,
in which every square is cartesian:
\[
\begin{tikzcd}
X^{(e)}_{\overline{\eta}} \ar[rr] \ar[d] & &  X \ar[d,hook] \\
A^{(e)}_{\overline{\eta}} \ar[r] \ar[d] &
\mathcal{A}^{(e)} \ar[d] \ar[r] & \mathbb{P} \\
\overline{\eta} \ar[r] & \mathcal{M}
\end{tikzcd}.
\]
For any perverse sheaf \(\mathcal{P}\) on \(X\), let
\(\mathcal{P}^{(e)}\) denote the inverse image of \(\mathcal{P}[-e]\) via
the morphism \(X^{(e)}_{\overline{\eta}} \to X\).  By
Lemma~\ref{lemma:general-hyperplane-preserve-perv},
\(\mathcal{P}^{(e)}\) is always a perverse sheaf.
\end{situation}

\begin{lemma}\label{lemma:katz}
Notation as in Construction~\ref{situation:flag-construction},
suppose \(X_0 \subset \mathbb{P}^{N}_{\mathbb{F}_q}\) is affine.
Let \((\mathcal{P}_{\alpha})_{\alpha \in I}\) be a weakly
\((E, I)\)-compatible system of perverse sheaves.  Then there exists a
collection of numbers \(P_{e} = P_{e}(X, (\mathcal{P}_{\alpha}))\) for
\(e = 0, 1, \ldots\), such that, for any \(e\), we have
\[
B_{c}(X^{(e)}_{\overline{\eta}}, \mathcal{P}_{\alpha}^{(e)})
\coloneq
\sum_{j \in \mathbb{Z}}
\dim \mathrm{H}^{j}_{c}(X^{(e)}_{\overline{\eta}}; \mathcal{P}^{(e)}_{\alpha})
\leq P_{e}
\]
for all \(\alpha \in I\).
\end{lemma}

\begin{proof}
We adapt Katz's argument based on the Euler
characteristic~\cite{katz_betti}, originally established for local
systems, to our setting of perverse sheaves.  For convenience, we
denote \(h^i_{c}(-) = \dim \mathrm{H}^{i}_{c}(-)\) throughout the
proof.

By constructibility of direct images, there exists an open dense
subscheme \(U(\alpha) \subset \mathcal{M}\) such that, for any flag
\[
A_0 \supset A_1 \supset \cdots \supset A_e \supset \cdots
\]
with \(A_{\bullet} \in U(\alpha)\), the Betti numbers
\[
h^i_{c}(X \cap A_e; \mathcal{P}_{\alpha}|_{X \cap A_e}[-e])
\]
are constant.  By the smooth and proper base change theorems, these
Betti numbers are equal to
\(h^i(X^{(e)}_{\overline{\eta}}; \mathcal{P}^{(e)}_{\alpha})\).

Now, for any \(\alpha, \beta \in I\), if \(A_{\bullet}\) lies in
\(U(\alpha) \cap U(\beta)\),
Lemma~\ref{lemma:euler-characteristic-independent-of-l} yields
\[
\chi(X \cap A_e; \mathcal{P}_{\alpha}|_{X \cap A_e}[-e]) =
\chi(X \cap A_e; \mathcal{P}_{\beta}|_{X \cap A_e}[-e])\,,
\]
which implies
\begin{equation}
\label{eq:generic-indept-ell-euler}
\chi(X^{(e)}_{\overline{\eta}}; \mathcal{P}^{(e)}_{\alpha}) =
\chi(X^{(e)}_{\overline{\eta}}; \mathcal{P}^{(e)}_{\beta})\,.
\end{equation}
In other words, the Euler characteristic of
\(\mathcal{P}^{(e)}_{\alpha}\) is independent of \(\ell\).

By Artin's vanishing theorem (Theorem~\ref{theorem:relative-affine-lefschetz}),
we have \(h^{i}_{c}(X^{(e)}_{\overline{\eta}}; \mathcal{P}^{(e)}_{\alpha}) = 0\) for all \(i < 0\),
so we only consider non-negative degrees below.

Next, the weak Lefschetz theorem (Lemma~\ref{lemma:gysin})
ensures that for each \(\alpha\) and a general flag \(A_{\bullet}\) in \(U(\alpha)\),
\begin{equation*}
\mathrm{H}^{j}_{c}(X \cap A_{e+1}; \mathcal{P}_{\alpha}|_{X\cap A_{e}}[-(e+1)](-1))
\to \mathrm{H}^{j+1}_{c}(X \cap A_{e};\mathcal{P}_{\alpha}[-e])
\end{equation*}
is surjective for \(j = 0\) and an isomorphism for \(j \geq 1\).
Thus, by smooth and proper base change, for all \(e \geq 0\),
\begin{equation}
\label{eq:chain-ineq}
h^{0}_{c}(X^{(e)}_{\overline{\eta}}; \mathcal{P}^{(e)}_{\alpha}) \geq
h^{1}_{c}(X^{(e-1)}_{\overline{\eta}}; \mathcal{P}^{(e-1)}_{\alpha})
\geq \cdots \geq h^{e}_{c}(X; \mathcal{P}_{\alpha})
\end{equation}
and, moreover,
\begin{equation}
\label{eq:chain-eq}
h^{j}_{c}(X^{(e)}_{\overline{\eta}}; \mathcal{P}^{(e)}_{\alpha}) =
h^{j+1}_{c}(X^{(e-1)}_{\overline{\eta}}; \mathcal{P}_{\alpha}^{(e-1)}) = \cdots =
h^{j+e}_{c}(X;\mathcal{P}_{\alpha}) \quad \text{for } j \geq 1\,.
\end{equation}

Applying \eqref{eq:chain-ineq} and \eqref{eq:chain-eq} for any \(e \geq 0\), we obtain
\begin{align*}
  \chi(X^{(e)}_{\overline{\eta}}; \mathcal{P}^{(e)}_{\alpha})
  &= h^{0}_{c}(X^{(e)}_{\overline{\eta}}; \mathcal{P}^{(e)}_{\alpha}) - h^{1}_{c}(X^{(e)}_{\overline{\eta}}; \mathcal{P}^{(e)}_{\alpha}) + h^{2}_{c}(X^{(e)}_{\overline{\eta}}; \mathcal{P}^{(e)}_{\alpha}) - \cdots \\
  &\geq h^{0}_{c}(X^{(e)}_{\overline{\eta}}; \mathcal{P}^{(e)}_{\alpha}) - h^{0}_{c}(X^{(e+1)}_{\overline{\eta}}; \mathcal{P}^{(e+1)}_{\alpha}) + h^{1}_{c}(X^{(e+1)}_{\overline{\eta}}; \mathcal{P}^{(e+1)}_{\alpha}) - \cdots \\
  &= h^{0}_{c}(X^{(e)}_{\overline{\eta}}; \mathcal{P}^{(e)}_{\alpha}) - \chi(X^{(e+1)}_{\overline{\eta}}; \mathcal{P}^{(e+1)}_{\alpha}).
\end{align*}
Thus, using \eqref{eq:chain-ineq}, we find
\[
h^{e}_{c}(X; \mathcal{P}_{\alpha})
\leq h^{0}_{c}(X^{(e)}_{\overline{\eta}}; \mathcal{P}^{(e)}_{\alpha})
\leq \chi(X^{(e)}_{\overline{\eta}}; \mathcal{P}^{(e)}_{\alpha}) + \chi(X^{(e+1)}_{\overline{\eta}}; \mathcal{P}^{(e+1)}_{\alpha})\,.
\]
In view of \eqref{eq:generic-indept-ell-euler},
this gives an \(\ell\)-independent upper bound for the dimension of
each cohomology group of \(\mathcal{P}_{\alpha}\). The proof is
complete.
\end{proof}

The following lemma should be well-known to the specialists.

\begin{lemma}\label{lemma:cc-independent-of-ell}
If \((\mathcal{F}_{\alpha})_{\alpha\in I}\) is a weakly
\((E,I)\)-compatible system on a smooth variety \(X_{0}\), then the
characteristic cycle \(\mathrm{CC}(\mathcal{F}_{\alpha})\) is
independent of \(\alpha\).
\end{lemma}

(A related compatibility,
namely \((\mathcal{F}_{\alpha})_{\alpha\in I}\) have
\emph{compatible ramification}, was shown in
\cite{lu-zheng_compatible-ramification}.)

\begin{proof}
Let \(h\colon W_0 \to X_0\) be any morphism.  Then, the system of
pullbacks \((h^*\mathcal{F}_\alpha)_{\alpha \in I}\) is again weakly
\((E, I)\)-compatible.  By
Lemma~\ref{lemma:euler-characteristic-independent-of-l}, for any
\(\alpha, \beta \in I\), the Euler characteristics agree:
\[
\chi(W, h^{*}\mathcal{F}_\alpha) = \chi(W, h^{*}\mathcal{F}_\beta).
\]
In other words, the complexes \(\mathcal{F}_\alpha\) and
\(\mathcal{F}_\beta\) have universally the same Euler characteristics,
in the sense of
\cite[Definition~3.1]{saito-yatagawa_wild-ramification-determines-characteristic-cycle}.

According to
\cite{saito-yatagawa_wild-ramification-determines-characteristic-cycle},
when the coefficient fields are finite, having universally the same
Euler characteristics implies that the characteristic cycles coincide.
To apply this result in our situation, we need to relate the
characteristic cycles defined with
\(\overline{\mathbb{Q}}_\ell\)-coefficients to those over finite
fields.

This connection is established by the ``decomposition homomorphism''
\(d_X\), constructed by Umezaki, Yang, and Zhao
\cite[\S5.2.6]{umezaki-yang-zhao}, which maps the Grothendieck group
of constructible \(\overline{\mathbb{Q}}_\ell\)-complexes to that of
constructible \(\overline{\mathbb{F}}_\ell\)-complexes.  These
decomposition homomorphisms are compatible with ordinary and
extraordinary direct image and inverse image functors
(\cite[\S5.2.7]{umezaki-yang-zhao}); thus, for any
\(\alpha, \beta \in I\), the Euler characteristics of the complexes
\(d_X\mathcal{F}_\alpha\) and \(d_X\mathcal{F}_\beta\) also agree
universally (cf.~the commutativity of the upper right diagram in
\cite[(5.2.7.1)]{umezaki-yang-zhao}).

Now, applying \cite[Proposition~3.4]{saito-yatagawa_wild-ramification-determines-characteristic-cycle}, we find that
\[
\begin{tikzcd}[row sep=small]
\mathrm{CC}(d_{X}\mathcal{F}_{\alpha}) \ar[r,equal] \ar[d,equal,"\text{\cite[5.{3}.2]{umezaki-yang-zhao}}",swap] & \mathrm{CC}(d_{X}\mathcal{F}_{\beta}) \ar[d,equal,"\text{\cite[5.{3}.2]{umezaki-yang-zhao}}"] \\
\mathrm{CC}(\mathcal{F}_{\alpha}) & \mathrm{CC}(\mathcal{F}_{\beta})
\end{tikzcd}
\]
for every pair \(\alpha, \beta \in I\).  In other words, the
characteristic cycle \(\mathrm{CC}(\mathcal{F}_{\alpha})\) is
independent of \(\alpha\), as claimed.
\end{proof}

\begin{lemma}\label{lemma:characteristic-numbers-independent-ell}
Let \(\mathbb{P}_{0}\) be a smooth projective variety over
\(\mathbb{F}_{q}\) together with a very ample invertible sheaf
\(\mathscr{L}\).  Let \((\mathcal{F}_{\alpha})_{\alpha\in I}\) be a
weakly \((E,I)\)-compatible system.  Then for each \(i \geq 0\), the
characteristic numbers
\[
c_{1}(\mathscr{L})^{i} \capprod \mathrm{cc}_{\mathbb{P}, i}(\mathcal{F}_{\alpha})
\quad (\alpha \in I)
\]
are independent of \(\ell\).
\end{lemma}

\begin{proof}
This follows from Lemma~\ref{lemma:cc-independent-of-ell} and the
definition of characteristic classes\(\mathrm{cc}\)
(cf.~\ref{situation:construction-cc}).  But there is also a much more
elementary argument for this easier statement.  By
\eqref{eq:characteristic-class-polynomial}, the characteristic numbers
are the coefficients of the polynomial
\[
P_{\alpha}(d) = \chi(Y_{d}, \mathcal{F}_{\alpha}|_{Y_{d}}[-1])
\]
where \(Y_{d} \subset \mathbb{P}\) is the hypersurface defined by a
sufficiently general section of \(\mathscr{L}^{\otimes d}\).  Such a
hypersurface may not be defined over \(\mathbb{F}_{q}\), but it exists
after performing a finite extension of \(\mathbb{F}_{q}\).  Since
\(P_{\alpha}(d)\) is independent of \(\ell\) by
Lemma~\ref{lemma:euler-characteristic-independent-of-l}, the
characteristic numbers are also independent of \(\ell\).
\end{proof}

\begin{theorem}\label{theorem:independent-of-ell-for-main}
Let \(\mathbb{P}_0\) be a smooth projective variety over
\(\mathbb{F}_{q}\).  Let \((\mathcal{P}_{\alpha})_{\alpha\in I}\) be a
weakly compatible \((E,I)\)-system of perverse sheaves on
\(\mathbb{P}_0\).  Set Then the constants \(C_{1}\), \(C_{2}\) and
\(C_{3}\) in Theorem~\ref{theorem:main} can be chosen to be
independent of \(\ell\).
\end{theorem}

\begin{proof}
By the explicit bounds in \ref{situation:explication-1},
\ref{situation:explication-2}, and \ref{situation:explication-3},
together with the \(\ell\)-independence of characteristic numbers
(Lemma~\ref{lemma:characteristic-numbers-independent-ell}), it
suffices to bound the numbers \(B(A_{e}, \mathcal{P}_{e})\)
independently of \(\ell\), where \(A_{\bullet}\) is a general flag and
\(\mathcal{P}_{e} = \mathcal{P}[-e]\).  Applying
Construction~\ref{situation:flag-construction} and corresponding
notation, it is enough to show that
\(B(X^{(e)}_{\overline{\eta}}, \mathcal{P}^{(e)}_{\alpha})\) admits a
uniform bound independent of \(\ell\).

Cover \(X_0\) with finitely many affine open subsets
\(X_0 = \bigcup_{\nu=1}^r U_{0\nu}\).  For any nonempty subset
\(J \subset \{1,\ldots,r\}\), let
\(U_{0,J} = \bigcap_{\nu\in J} U_{0,\nu}\).  The Mayer--Vietoris
spectral sequence reads:
\[
E_2^{-a,b} = \bigoplus_{\#J = a+1} \mathrm{H}_c^b(U^{(e)}_{J, \overline{\eta}}; \mathcal{P}^{(e)}_\alpha)
\Longrightarrow \mathrm{H}_c^{b-a}(X^{(e)}_{\overline{\eta}}; \mathcal{P}^{(e)}_\alpha).
\]
where \(a,b\geq 0\).

Each \(U_J\) is affine, so by Lemma~\ref{lemma:katz}, there is an
upper bound
\[
B_c(U^{(e)}_{J, \overline{\eta}}; \mathcal{P}^{(e)}_\alpha) \leq M_J
\]
with \(M_J\) independent of \(\alpha\in I\).  Thus,
\[
B_c(X^{(e)}_{\overline{\eta}}; \mathcal{P}_\alpha) \leq \sum_{\substack{J \subset \{1,\ldots,r\} \\ J \neq \emptyset}} M_J,
\]
which gives an \(\ell\)-independent upper bound as required.
\end{proof}

The same argument yields the following result,
which is of independent interest.
Katz proved a similar result for lisse sheaves using above mentioned method
\cite{katz_betti}, so we make no claim to originality here.

\begin{theorem}\label{theorem:katz}
Let \(X_0\) be a separated scheme of finite type over
\(\mathbb{F}_{q}\).  Let \((\mathcal{P}_{\alpha})_{\alpha\in I}\) be a
weakly \((E,I)\)-compatible system of perverse sheaves on \(V\).
\begin{enumerate}
\item There exists a constant
\(M = M(X,(\mathcal{P}_{\alpha})_{\alpha\in I}) > 0\) independent of
\(\ell\), such that
\[
B_{c}(X;\mathcal{P}_{\alpha}) = \sum_{i} \dim
\mathrm{H}_{c}^{i}(X; \mathcal{P}_{\alpha}) \leq M
\]
for all \(\alpha \in I\).
\item There exists a constant
\(N = N(X,(\mathcal{P}_{\alpha})_{\alpha\in I})\), such that
\[
B(X,\mathcal{P}_{\alpha}) \leq N.
\]
for any \(\alpha\in I\).
\end{enumerate}
\end{theorem}

\begin{proof}
It suffices to prove (1), because a theorem of
Gabber~\cite[Theorem~2]{fujiwara_gabber-independence-of-l-for-intersection-cohomology}
states that if \((\mathcal{F}_{\alpha})_{\alpha\in I}\) is weakly
\((E,I)\) compatible if and only if
\((\mathbb{D}_{X}\mathcal{F}_{\alpha})_{\alpha\in I}\) is.
To prove (1), a Mayer--Vietoris argument as above reduces to the case where \(X\) is affine.  For the affine case one appleis Lemma~\ref{lemma:katz}
\end{proof}

\begin{corollary}\label{corollary:independent-of-ell-for-arbitrary}
Suppose \(X_0\) is a projective variety over a finite field
\(\mathbb{F}_q\), and suppose \(\mathcal{F}_{\alpha}\) is a weakly
\((E,I)\)-compatible system of constructible sheaves on \(X_0\).  Then
the constant \(C\) in Theorem~\ref{theorem:arbitrary} can be made
independent of \(\ell\).
\end{corollary}

\begin{proof}
Without loss of generality, we may assume that all irreducible
components of \(X_0\) are defined over \(\mathbb{F}_q\).  Thus, all
auxiliary subvarieties and open subsets constructed in the proof of
Theorem~\ref{theorem:arbitrary} (namely \(X^\circ\), \(D^\circ\),
etc.) are intrinsically defined in terms of \(X_0\), and are
independent of the choice of~\(\ell\).

By construction, the complexes \(j_! j^* \mathcal{F}_\alpha[n]\)
(where \(j\colon X^\circ \hookrightarrow X_0\) and \(n = \dim X_0\))
and \(\mathcal{F}_\alpha|_D\) form weakly \((E,I)\)-compatible systems
(see the notation in the proof of Theorem~\ref{theorem:arbitrary}).
Applying Theorem~\ref{theorem:independent-of-ell-for-main} in this
context shows that the bounds for the Betti numbers of
\(j_! j^* \mathcal{F}_\alpha[n]\) are independent of~\(\ell\).

Once the Betti numbers for the open part have been bounded, we pass to
the complement \(D\), which has smaller dimension and again satisfies
the same hypotheses.  Applying the same argument recursively to each
such \(D\) and its analogues at each step, we find at every stage
compatible systems of sheaves on varieties defined over
\(\mathbb{F}_q\), to which
Theorem~\ref{theorem:independent-of-ell-for-main} applies.  At each
step, the resulting constants are independent of~\(\ell\), since all
relevant data and decompositions are defined over \(\mathbb{F}_q\) and
the arguments depend only on the compatible systems.

After finitely many steps, this process yields the result for all of
\(X_0\), and the uniform bound follows.  Thus, the constant \(C\) in
Theorem~\ref{theorem:arbitrary} may be chosen independently
of~\(\ell\).
\end{proof}

By a standard argument (extracting coefficients, and spread out),
using that the constant sheaf is always in a compatible system,
we deduce the following.

\begin{corollary}
Let \(X\) be any variety over an algebraically closed field \(k\)
and \(\mathcal{E} = \Lambda_X\).
Then the constants in Theorem~\ref{theorem:arbitrary} and Theorem~\ref{theorem:lci} can be chosen to be independent of \(\ell\).
\end{corollary}

We conclude this paper with a open problem.

\begin{situation}
[Problem]
\textit{Suppose \((\mathcal{F}_{\alpha})_{\alpha\in I}\) is a weakly
  \((E,I)\)-compatible system of constructible complexes on a smooth
  variety \(\mathbb{P}_{0}\) supported on the same \(n\)-dimensional
  subvariety \(S_{0} \subset \mathbb{P}_{0}\).  Find an upper bound of the
  form
\[
B(Y, \mathcal{F}_{\alpha}) \leq A \cdot d^n + B \cdot d^{n-1}
\quad (Y \text{ an arbitrary degree } d \text{ hypersurface in } \mathbb{P})
\]
where \(A\) and \(B\) are constants independent of \(\alpha\), \(Y\), and \(d\).}
\end{situation}

Although we have obtained an upper bound in Corollary~\ref{corollary:bound-any},
it is not clear whether the constant \(C\) appearing in this bound is independent
of~\(\ell\).  This issue arises because it is not known whether applying
\({}^{\mathrm{p}}\mathcal{H}^{e}\) to a weakly \((E, I)\)-compatible system
of constructible complexes on a smooth variety \(\mathbb{P}\) produces a
weakly \((E, I)\)-compatible system of perverse sheaves on \(\mathbb{P}\).

On the positive side, when the \(\mathcal{F}_{\alpha}\)
are honest constructible sheaves
(rather than complexes),
Theorem~\ref{theorem:arbitrary} offers an alternative approach that preserves
compatibility and avoids the need to take perverse cohomology.
Therefore, Corollary~\ref{corollary:independent-of-ell-for-arbitrary} gives
a positive answer to the question in the case where the
\(\mathcal{F}_{\alpha}\) are
constructible sheaves.

\printbibliography
\end{document}